\documentclass[11pt,a4paper]{article}
\usepackage[T1]{fontenc}
\usepackage[utf8]{inputenc}
\usepackage{authblk}
\usepackage[margin=1in]{geometry}
\usepackage{cite,indentfirst,amsmath,amssymb,mathrsfs,booktabs,multirow,amsfonts}
\usepackage{indentfirst,graphicx,subfig,verbatim,enumerate,enumitem} 
\newtheorem{theorem}{Theorem}[section]
\newtheorem{lemma}{Lemma}[section]
\newtheorem{corollary}{Corollary}[section]

\newenvironment{proof}{{\noindent\it Proof:}}{\hfill $\square$\par}

\title{ The stabilized exponential-SAV approach preserving maximum bound principle
for nonlocal Allen-Cahn equation }
\author[a]{Xiaoqing Meng}
\author[a]{Aijie Cheng \thanks{Corresponding author: aijie@sdu.edu.cn}}
\author[b]{Zhengguang Liu}
\affil[a]{School of Mathematics, Shandong University, Jinan, Shandong 250100, China.}
\affil[b]{School of Mathematics and Statistics, Shandong Normal University, Jinan, Shandong 250014, China.}

\date{}

\begin{document}
\maketitle 

\begin{abstract}
The nonlocal Allen-Cahn equation with nonlocal diffusion operator 
is a generalization of the classical Allen-Cahn equation. 
It satisfies the energy dissipation law and maximum bound principle (MBP), 
and is important for simulating a series of physical and biological phenomena 
involving long-distance interactions in space.
In this paper, we construct first- and second-order (in time) 
accurate, unconditionally energy stable and MBP-preserving schemes
for the nonlocal Allen-Cahn type model based on the 
stabilized exponential scalar auxiliary variable (sESAV) approach. 
On the one hand, we have proved the MBP and unconditional energy stability 
carefully and rigorously in the fully discrete levels. 
On the other hand, we adopt an efficient FFT-based fast solver to 
compute the nearly full coefficient matrix generated from the spatial discretization, 
which improves the computational efficiency.
Finally, typical numerical experiments are presented to demonstrate the 
performance of our proposed schemes. 

\noindent{\textbf{keywords}} 
Nonlocal Allen-Cahn equation; Exponential scalar auxiliary variable approach;
Unconditional energy stability; Maximum bound principle; Fast solver;
\end{abstract}
\thispagestyle{empty}

\section{Introduction}
The mesoscopic description of phase transition is usually based on 
the assumption that the evolution of phase variable $\phi$ follows 
a gradient flow with free energy relative to a certain metric.
In general, from a mathematical point of view, phase field models 
are always derived from functional variations of free energy \cite{33}.
The classical Allen-Cahn equation was proposed by Allen and 
Cahn (1979) \cite{1} to simulate the motion of anti-phase boundaries in 
crystalline solids. 
The $L^2$ gradient flow with respect to the 
following free energy functional \cite{10}: 
\begin{align*}
    E\left( \phi \right) =\int_{\varOmega}{\left( \frac{\varepsilon ^2}{2}\left| \nabla \phi \left( \mathbf{x} \right) \right|+F\left( \phi \left( \mathbf{x} \right) \right) \right)}d\mathbf{x}.
\end{align*}
is the classical Allen-Cahn equation:
\begin{align}
    \label{LAC}
    \frac{\partial \phi}{\partial t}=\varepsilon ^2\varDelta \phi -F'\left( \phi \right) ,\quad \left( \mathbf{x},t \right) \in \varOmega \times J,
\end{align}
with the corresponding initial condition and periodic or homogeneous Neumann boundary conditions,
where $\varOmega \subset \mathbb{R}^d $ is a bounded spatial domain, 
the time domain is $J=[0,T]$ with $T$ the final time, 
$\phi =\phi \left( \mathbf{x},t \right) :\bar{\varOmega }\times J\rightarrow \mathbb{R}$ 
is the phase variable, $\varepsilon >0$ denotes the interface width,
and $F(\phi)$ is energy density function.
In the classical Allen-Cahn model, the thickness of the interface is modeled 
as order $\varepsilon $, where $0<\varepsilon \ll 1$. 
At the interface layer, the value of phase variable $\phi$ changes rapidly.\par
Roughly speaking, replacing the Laplacian $\Delta $ in the equation \eqref{LAC} by the nonlocal
diffusion operator $-\mathcal{L}$ \cite{3}, one can obtain the 
nonlocal Allen-Cahn (NAC) equation as follows \cite{4}:
\begin{align}
    \label{eqNAC}
    \frac{\partial \phi}{\partial t}=-\varepsilon ^2\mathcal{L}\phi -F'\left( \phi \right) ,\quad \left( \mathbf{x},t \right) \in  \varOmega \times J,
\end{align}
with the corresponding initial condition and periodic or homogeneous Neumann boundary conditions,
where the nonlocal operator $\mathcal{L}$ has been introducted by many articles 
such as \cite{4}:
\begin{align}
    \label{L}
    \mathcal{L}\phi \left( \mathbf{x} \right) =\int_{\varOmega}{J\left( \mathbf{x}-\mathbf{y} \right) \left[ \phi \left( \mathbf{x} \right) -\phi \left( \mathbf{y} \right) \right]}d\mathbf{y},\quad \mathbf{x}\in \varOmega ,
\end{align}
and the kernel $J$ satisfies the following conditions \cite{5}:
\begin{enumerate}[itemindent=1em]
    \item[1)] $J(\mathbf{x})\geq 0$ for any $\mathbf{x}\in \varOmega $,
    \item[2)] $J(\mathbf{x})=J(-\mathbf{x}) $,
    \item[3)] $J$ is $\varOmega $-periodic,
    \item[4)] $J$ is integrable, $J * 1=\int_{\varOmega}{J\left( \mathbf{x} \right)} d \mathbf{x}>0$ is a positive constant. 
\end{enumerate}\par
As mentioned in \cite{2,4,51,53,54,56}, the Allen-Cahn type equation
is important for simulating 
a variety of physical and biological phenomena involving media 
whose properties vary with space. 
It is well known that the classical Allen-Cahn equation  
is one of the typical systems of the phase filed models and 
has successfully applied to many complex interface problems.
In recent years, the NAC equation for phase separations 
within the interaction kernels \cite{8,52} 
has attracted more and more attention, 
and has been applied to many fields such as 
phase transition \cite{6}, image processing \cite{48,49} and so on. 
The most important point is that no matter the classical Allen-Cahn 
or NAC equation, their exact solutions satisfy two basic physical properties, 
namely the maximum bound principle (MBP) \cite{7,8,10} 
and the energy dissipation law \cite{9,10}.
Therefore, an important issue is to develop algorithms that inherit 
these two properties at the discrete level.\par 
The main challenge for numerical solutions of phase field models 
is the time-stepping scheme, that is, designing appropriate methods 
to discretize nonlinear and nonlocal terms while preserving 
energy stability and MBP. 
There have been extensive study and application of  
energy stable numerical schemes for time discretization of phase field models,
such as convex splitting schemes \cite{11,12,13}, stabilized semi-implicit 
schemes \cite{4,14,15,16}, the exponential time differencing (ETD) schemes
\cite{8,18,19}, the Invariant Energy Quadratization (IEQ) \cite{20,21,22}, 
the Scalar Auxiliary Variable (SAV) methods \cite{9,23,25}, and 
many variants of SAV schemes\cite{26,28,29,30,31,32,33,34}.
However, these existing SAV-type numerical schemes can not guarantee the MBP.\par
In recent years, the MBP preservation has attracted considerable 
attention in the numerical methods community.
The semi-implicit schemes for the classic and fractional  
Allen-Cahn equations have been widely studied \cite{35,36,37,38,39,40}.
The first- and second-order stablilized ETD schemes can be used to 
unconditionally guarantee the MBP of the NAC equation \cite{8} 
and the conservative Allen-Cahn equation \cite{41}.
In addition, the third- and fourth-order stabilized integration factor Runge-Kutta (IFRK) scheme are 
proposed \cite{42,43}.\par
In fact, the classical Allen-Cahn model can actually 
be viewed as the approximation of the NAC model where the nonlocal 
diffusion operator is replaced by the Laplace operator. 
The main objective of this paper is to develop and analyze an accurate 
and efficient linear algorithm based on the idea of 
stabilized exponential-SAV (sESAV) \cite{10} method 
by introducing an appropriate stability term 
for solving the general NAC models with general nonlinear potential. 
The MBP and unconditional energy stability of the fully discrete scheme 
are proved, and it can be seen that the first-order scheme 
satisfies MBP unconditionally, while the second-order scheme 
by Crank-Nicolson approximation has constrained condition for time step.
Furthermore, we use the second-order central finite difference 
in space, and due to the large amount of computation and memory required to 
solve this linear system, we analyze the properties of the stiffness matrix 
and use a fast and efficient solution method \cite{45} based on FFT to 
reduce the CPU computation time and memory requirements.\par
The rest of this paper is organized as follows. 
In section 2, we briefly introduce the NAC model with 
general nonlinear potential and give the spatial discretization 
formula of the nonlocal operator. 
In section 3, we construct the first- and second-order sESAV 
schemes for the NAC equation and prove the 
MBP and unconditional energy stability in the fully discrete level.
In section 4, We demonstrate the effectiveness, stability 
and accuracy of the proposed scheme by using the fast solution method 
through several 2D numerical examples.

\section{The NAC model and spatial discretization}
In this section, we introduce the NAC model with general nonlinear potential 
using the energy variational approach, and give some symbols that will be 
used for subsequent analysis. 
For simplicity, the two-dimensional domain is
$\varOmega =\left[ -L_x,L_x \right] \times \left[ -L_y,L_y \right] $, 
a rectangular cell in $\mathbb{R}^2$.
We use the second-order central difference scheme for spatial discretization, 
and for other spatial discretization, we refer to \cite{7,8,45,55} for more details.
\subsection{The NAC model}
Let us consider the nonlocal free energy functional \cite{5}:
\begin{align}
    \label{NE}
    E\left( \phi \right) =\int_{\varOmega}{\left( \frac{\varepsilon ^2}{4}\int_{\varOmega}{J\left( \mathbf{x}-\mathbf{y} \right) \left[ \phi \left( \mathbf{x} \right) -\phi \left( \mathbf{y} \right) \right]^2}d\mathbf{y}+F\left( \phi \left( \mathbf{x} \right) \right) \right)}d\mathbf{x}.
\end{align}
From the definition of the nonlocal operator $\mathcal{L}$ and 
the conditions that the kernel $J$ satisfies, we can obtain
\begin{align}
    \left( \mathcal{L}\phi ,\phi \right) =\frac{1}{2}\int_{\varOmega}{\int_{\varOmega}{J\left( \mathbf{x}-\mathbf{y} \right) \left[ \phi \left( \mathbf{x} \right) -\phi \left( \mathbf{y} \right) \right]^2}\text{d}\mathbf{y}}\text{d}\mathbf{x}\ge 0.
    \label{LJ}
\end{align}
Then the nonlocal free energy functional \eqref{NE} can be equivalently 
written as follows:
\begin{align}
    \label{E}
    E\left( \phi \right) =\frac{\varepsilon ^2}{2}\left( \mathcal{L}\phi ,\phi \right) +\int_{\varOmega}{F\left( \phi \left( \mathbf{x} \right) \right)}\text{d}\mathbf{x},
\end{align}
In fact, the $L^2$ gradient flow for the energy functional \eqref{E} is the NAC equation \cite{8}:
\begin{align}
    \label{NAC}
    \frac{\partial \phi}{\partial t}=-\varepsilon ^2\mathcal{L}\phi +f\left( \phi \right) ,\quad \left( \mathbf{x},t \right) \in \varOmega \times J,
\end{align}
where $f(\phi)=-F'(\phi)$, the nonlocal diffusion operator $\mathcal{L}$ is a linear, 
self-adjoint, positive semi-definite operator.
Just by simple derivation, we can get that
for any function $\eta \left( \mathbf{x} \right)$ which is sufficiently smooth and satisfies
$\eta |_{\partial \varOmega}=0$ on $\partial \varOmega $, 
the functional derivative 
$\frac{\delta E\left( \phi \right)}{\delta \phi}$ 
can be expressed by the following formula
\begin{align*}
     \int_{\varOmega}{\frac{\delta E\left( \phi \right)}{\delta \phi}\eta}d\mathbf{x}
    ={}& \lim_{\theta \rightarrow 0}\frac{E\left( \phi +\theta \eta \right) -E\left( \phi \right)}{\theta}\\    
    ={}& \lim_{\theta \rightarrow 0}\frac{1}{\theta}\int_{\varOmega}{\left( \frac{\varepsilon ^2}{2}\left( \left( \mathcal{L}\phi ,\phi \right) +2\theta \left( \mathcal{L}\phi ,\eta \right) +\theta ^2\left( \mathcal{L}\eta ,\eta \right) \right) -\frac{\varepsilon ^2}{2}\left( \mathcal{L}\phi ,\phi \right) \right)}d\mathbf{x}\\
    & +\lim_{\theta \rightarrow 0}\int_{\varOmega}{\frac{F\left( \phi +\theta \eta \right) -F\left( \phi \right)}{\theta}}d\mathbf{x}\\
    ={}& \varepsilon ^2\left( \mathcal{L}\phi ,\eta \right) +\int_{\varOmega}{F'\left( \phi \right) \eta}d\mathbf{x}\\
    ={}& \int_{\varOmega}{\left( \varepsilon ^2\mathcal{L}\phi +F'\left( \phi \right) \right) \eta}d\mathbf{x},
\end{align*}
then we have the NAC equation
$$
\frac{\partial \phi}{\partial t}=-\frac{\delta E\left( \phi \right)}{\delta \left( \phi \right)}=-\left( \varepsilon ^2\mathcal{L}\phi +F'\left( \phi \right) \right) 
,$$
with the initial condition
\begin{align*}
    \phi \left( \mathbf{x},0 \right) =\phi _0\left( \mathbf{x} \right) ,\quad \mathbf{x}\in \bar{\varOmega} ,
\end{align*}
and periodic or homogeneous Neumann boundary conditions.\par
According to \cite{4,7,8,45}, we know that the NAC model satisfies
the energy dissipation law, which means 
the free energy functional \eqref{E} corresponding to the exact solution of 
the NAC equation \eqref{eqNAC} decreases with time, that is,
\begin{align}
    \label{dis}
    \frac{dE\left( \phi \left( t \right) \right)}{dt}\le 0.
\end{align}
It has been proved in \cite{4,43} that the NAC equation \eqref{eqNAC} satisfies the MBP in the sense that 
if the maximum absolute value of the initial value $\phi_0$ is bounded by $\beta$, 
then the maximum absolute value of the solution $\phi$ is always bounded 
by $\beta$ for all time, i.e.,
\begin{align}
    \label{MBP}
    \underset{\mathbf{x}\in \bar{\varOmega }}{\max}\left| \phi _0 \right|\le \beta \Rightarrow \max_{\mathbf{x}\in \bar{\varOmega }}\left| \phi \left( t \right) \right|\le \beta ,\quad \forall t\in J.
\end{align}
\subsection{Spatial discretization}
We use the second-order central difference approach to discretize the 
space variables. Let $N_x$,$N_y,N_t$ be positive integers, and 
we divide both the spatial domain and the time domain into uniform grids 
\begin{equation*}
    h_x=\dfrac{2L_x}{N_x},\quad h_y=\dfrac{2L_y}{N_y},\quad \tau=\dfrac{T}{N_t}.
\end{equation*}
Then we define the following uniform grids:
\begin{align*}
    &\varOmega _h=\left\{ \left( x_i,y_j \right) \left| x_i=-L_x+i h_x,y_j=-L_y+j h_y,0\le i\le N_x,0\le j\le N_y \right. \right\}     ,\\
    &\varOmega _t=\left\{ t^n\left| t^n=n\tau ,0\le n \le N_t \right. \right\} .
\end{align*}
\par
The $L^2$ inner product between functions $\varphi \left( \mathbf{x} \right)$ 
and $\psi \left( \mathbf{x} \right)$ is 
denoted by 
$$
\left( \varphi ,\psi \right) =\int_{\varOmega}
{\varphi \left( \mathbf{x} \right) \psi \left( \mathbf{x} \right) 
\,\text{d}\mathbf{x}},\quad \forall \varphi ,\psi \in L^2\left( \varOmega \right) .
$$
Let $\mathcal{M}_h$ be the set of all periodic grid functions over $\varOmega_h $.
The discrete $L^2$ inner product $\left< \cdot ,\cdot \right> $, discrete $L^2$
norm $\lVert \cdot \rVert $, and discrete $L^\infty $ norm can be defined as:
$$
\left< \varphi , \psi \right> =h_xh_y\sum_{i=0}^{N_x}{\sum_{j=0}^{N_y}{\varphi _{i,j}}}\psi _{i,j},\quad
\lVert \varphi \rVert =\sqrt{\left< \varphi ,\varphi \right>},\quad
\lVert \varphi \rVert _{\infty}=\max_{0\le i\le N_x,0\le j\le N_y}\left| \varphi _{i,j} \right|,\quad
\forall \varphi ,\psi \in \mathcal{M}_h.
$$
At the same time, the spatically-discretized original energy functional can be defined as:
\begin{align}
    \label{DE}
    E_h\left( \omega \right) :=\frac{\varepsilon ^2}{2}\left< \mathcal{L}_h\omega ,\omega \right> +\left< F\left( \omega \right) ,1 \right> ,\quad \omega \in \mathcal{M}_h.
\end{align} 
\par
From \cite{33,34}, we can write the nonlocal operator $\mathcal{L}$ \eqref{L} 
as the following equivalent form:
\begin{equation*}
    \mathcal{L}\phi =\left( J*1 \right) \phi -J*\phi,
\end{equation*}
where
\begin{align*}
    &J*1=\int_{\varOmega}{J\left(\mathbf{x}\right)\,\text{d}\mathbf{x}},\\
    &\left(J*\phi\right)\left(\mathbf{x}\right)
    =\int_{\varOmega}{J\left(\mathbf{x}-\mathbf{y}\right)}\phi\left(\mathbf{y}\right)\,\text{d}\mathbf{y}
    =\int_{\varOmega}{J\left(\mathbf{y}\right)}\phi\left(\mathbf{x}-\mathbf{y}\right)\,\text{d}\mathbf{y},
\end{align*}
are exactly the periodic convolutions.
For any $\phi$, we can discrete $\mathcal{L}\phi$ at 
$\left( t^n,x_i,y_j \right) \in \varOmega _t\times \varOmega _h$, 
as follows:
\begin{align*}
    \left( \mathcal{L}_h\phi  \right) _{i,j}^n=\left( J*1 \right)_{i,j} \phi _{i,j}^n-\left( J*\phi  \right) _{i,j}^n,
    \quad 0\le i\le N_x,\quad 0\le j\le N_y,\quad 0\le n\le N_t,
\end{align*}
where the discrete formation $\left( J*1 \right)_{i,j} \phi _{i,j}^n$ and $\left( J*\phi  \right) _{i,j}^n$ can be written as:
\begin{align*}
    \left( J*1 \right) _{i,j}\phi _{i,j}^{n}=& h_xh_y\left[ \sum_{m_1=1}^{N_x-1}{\sum_{m_2=1}^{N_y-1}{J\left( x_{m_1}-x_i,y_{m_2}-y_j \right)}} \right.\\ 
        &+\frac{1}{2}\sum_{m_1=1}^{N_x-1}{\left( J\left( x_{m_1}-x_i,y_0-y_j \right) +J\left( x_{m_1}-x_i,y_{N_y}-y_j \right) \right)}\\
        &+\frac{1}{2}\sum_{m_2=1}^{N_y-1}{\left( J\left( x_0-x_i,y_{m_2}-y_j \right) +J\left( x_{N_x}-x_i,y_{m_2}-y_j \right) \right)}\\
        &+\frac{1}{4}\left( J\left( x_0-x_i,y_0-y_j \right) +J\left( x_{N_x}-x_i,y_0-y_j \right) \right)\\
        &\left. +\frac{1}{4}\left( J\left( x_0-x_i,y_{N_y}-y_j \right) +J\left( x_{N_x}-x_i,y_{N_y}-y_j \right) \right) \right] \phi _{i,j}^{n},           
\end{align*}            
and
\begin{align*}            
    \left( J*\phi  \right) _{i,j}^n=
        &h_xh_y\left[ \sum_{m_1=1}^{N_x-1}{\sum_{m_2=1}^{N_y-1}{J\left( x_{m_1}-x_i,y_{m_2}-y_j \right) \phi _{m_1,m_2}^{n}}} \right.\\ 
        &+\frac{1}{2}\sum_{m_1=1}^{N_x-1}{\left( J\left( x_{m_1}-x_i,y_0-y_j \right) \phi _{m_1,0}^{n}+J\left( x_{m_1}-x_i,y_{N_y}-y_j \right) \phi _{m_1,N_y}^{n} \right)}\\
        &+\frac{1}{2}\sum_{m_2=1}^{N_y-1}{\left( J\left( x_0-x_i,y_{m_2}-y_j \right) \phi _{0,m_2}^{n}+J\left( x_{N_x}-x_i,y_{m_2}-y_j \right) \phi _{N_x,m_2}^{n} \right)}\\
        &+\frac{1}{4}\left( J\left( x_0-x_i,y_0-y_j \right) \phi _{0,0}^{n}+J\left( x_{Nx}-x_i,y_0-y_j \right) \phi _{N_x,0}^{n} \right)\\
        &\left. +\frac{1}{4}\left( J\left( x_0-x_i,y_{N_y}-y_j \right) \phi _{0,N_y}^{n}+J\left( x_{N_x}-x_i,y_{N_y}-y_j \right) \phi _{N_x,N_y}^{n} \right) \right].
\end{align*}
\par
Without loss of generality, we set $L_x=L_y,h_x=h_y=h,N_x=N_y=N$.
By applying four transformation operators 
$\mathcal{A}_1,\mathcal{A}_2,\mathcal{A}_3,\mathcal{A}_4$
and all elements of them are non-negative.
We transform the stiffness matrix into a block-Toeplitz-Toeplitz-block (BTTB) 
$(N+1)^2\times (N+1)^2$-matrix $\mathbf{B}$, 
and we can refer to reference \cite{45} for specific details on this step.
Then, for any $(N+1)^2$-vector $\omega$, we obtain the matrix-vector 
multiplication $\mathcal{L}_h\omega$ by 
\begin{align*}
    \quad\mathcal{L}_h\omega &=\frac{1}{4}h^2\left[ \mathbf{B}\left( \mathcal{A}_1+\mathcal{A}_2+\mathcal{A}_3+\mathcal{A}_4 \right) \mathbf{I}_1 \right] \cdot \omega -\frac{1}{4}h^2\mathbf{B}\left[ \left( \mathcal{A}_1+\mathcal{A}_2+\mathcal{A}_3+\mathcal{A}_4 \right) \omega \right] \\
    &=\frac{1}{4}h^2\left[ \left( \mathbf{B}_1\mathbf{I}_1 \right) \cdot \omega -\mathbf{B}_1\omega \right] \\
    &=\frac{1}{4}h^2\left( \mathbf{B}_2-\mathbf{B}_1 \right) \omega 
\end{align*}
where the symbol $\cdot$ in this formula represents the 
dot product between vectors,
the vertor $\mathbf{I}_1$ is all one $(N+1)^2$-vector, $(N+1)^2$-matrix $\mathcal{L}_h=-\frac{1}{4}h^2\left( \mathbf{B}_2-\mathbf{B}_1 \right)$, 
the matrix $\mathbf{B}_1=\mathbf{B}\left( \mathcal{A}_1+\mathcal{A}_2+\mathcal{A}_3+\mathcal{A}_4 \right)$,
and the diagonal matrix $\mathbf{B}_2$ is a $(N+1)^2 \times (N+1)^2$-matrix with the elements of $(N+1)^2$-vector 
$\mathbf{B}_1\mathbf{I}_1$ as the main diagonal elements.
As mentioned by \cite{45}, we know that matrix $\mathcal{L}_h$ is almost full,
so we use a fast solver by means of the techniques of FFT and FCG in the numerical experiments. 
By analyzing the elements of the discrete operator $\mathcal{L}_h$,
the following lemma is valid \cite{39}: 
\begin{lemma}
    \label{l1}
    For any $a>0$, we have $\lVert \left( aI+\mathcal{L}_h \right) ^{-1} \rVert _{\infty}\le a^{-1}$,
    where $I$ is the identity matrix. 
\end{lemma}
\begin{proof}
    From the properties of the kernel $J$ and the definition of the four 
    transformation matrixs, we know that all the elements of 
    $\mathbf{B},\mathcal{A}_1,\mathcal{A}_2,\mathcal{A}_3,\mathcal{A}_4$ 
    are non-negative, so the matrix $\mathcal{L}_h=(d_{i,j})$ 
    has the property that the main diagonal elements 
    are nonnegative and the other elements are nonpositive.
    Furthermore, it satisfies
    $$
    -d_{i,i}=\sum_{1\le j\le \left( N+1 \right) ^2,j\ne i}{d_{i,j}},\quad 1\le i\le \left( N+1 \right) ^2.
    $$
    Then we obtain that the matrix $aI+\mathcal{L}_h$ is strictly diagonally 
    dominant, so the matrix is invertible.
    Now we rewrite $aI+\mathcal{L}_h$ as:
    $$aI+\mathcal{L}_h=\left( a+d \right) \left( I-sC \right) ,$$
    where $d=\underset{1\le i\le \left( N+1 \right) ^2}{\max}d_{i,i}, 
    s=\frac{d}{a+d}<1, C=I-\frac{1}{d}\mathcal{L}_h$,
    and the matrix $C=(c_{i,j})$ satisfies 
    $$
    \lVert C \rVert _{\infty}=\max_{1\le i\le \left( N+1 \right) ^2}\sum_{j=1}^{\left( N+1 \right) ^2}{\left| c_{i.j} \right|}=\max_{1\le i\le \left( N+1 \right) ^2}\left( 1-\frac{d_{i,i}}{d}-\sum_{j=1,j\ne i}^{\left( N+1 \right) ^2}{\frac{d_{i,j}}{d}} \right) =1
    .$$
    By Gershgorin's circle theorem, we know that 
    $$
    \rho \left( sC \right) =s\rho \left( C \right) \le s<1,
    $$
    where $\rho\left( C \right)$ is the spectral radius of matrix $C$.
    So we have
    \begin{align*}
        \lVert \left( aI+\mathcal{L}_h \right) ^{-1} \rVert _{\infty}&=\lVert \left( \left( a+d \right) \left( I-sC \right) \right) ^{-1} \rVert _{\infty}=\frac{1}{a+d}\lVert \sum_{p=0}^{\infty}{\left( sC \right) ^p} \rVert _{\infty}\\
        \le \frac{1}{a+d}\sum_{p=0}^{\infty}{s^p\lVert C \rVert _{\infty}^{p}}&= \frac{1}{a+d}\cdot \frac{1}{1-s}=\frac{1}{a}.
    \end{align*}
\end{proof}
\par
We assume that $f(\phi)=-F'(\phi)$ is continuously differentiable and
there exists a constant $\beta>0$ which satisfies formula \eqref{MBP}
such that 
$f\left( \beta \right) \le 0\le f\left( -\beta \right)$, 
then we have the lemma as follows \cite{7}: 
\begin{lemma}
    \label{l2}
    If for some positive constant $\kappa$, $\kappa \ge \lVert f' \rVert _{C\left[ -\beta ,\beta \right]}
    =\underset{\xi \in \left[ -\beta ,\beta \right]}{\max}\left| f'\left( \xi \right) \right|$,
    we have $$
    \left| f\left( \xi \right) +\kappa \xi \right|\le \kappa \beta ,\quad \forall \xi \in \left[ -\beta ,\beta \right].
    $$
\end{lemma}

\section{The sESAV schemes for the NAC model}
In this section, we mainly study the temporal discretization scheme of the NAC equation, 
and prove the energy stable and MBP of the fully discrete schemes based 
on the second-order central difference formula in space. 
In summary, we construct the linear, first- and second-order (in time) 
unconditionally energy stable and MBP-preserving schemes 
for the NAC model with general nonlinear potential.
The stabilized exponential-SAV (sESAV) \cite{33} 
is proposed to preserve the energy dissipation law and the MBP of the Allen-Cahn 
type equation, which is of great significance for solving the stiffness 
problem with thin interface. 
Thanks to the continuity of $F(\phi)$ and the MBP \eqref{MBP}, we know that
$$
E_1\left( \phi \right) :=\int_{\varOmega}{F\left( \phi \left( \mathbf{x} \right) \right)}\text{d}\mathbf{x}\ge -C_*,
$$
for some constant $C_*\ge 0$.
For the stabilized exponential-SAV (sESAV) scheme, 
similar to \cite{10}, we first introduct a scalar auxiliary variable: 
$s(t) = E_1(\phi(t)),$
so there is the modified nonlocal free energy functional 
which is equivalent to the initial energy functional \eqref{NE} 
in the continuous sense:
\begin{align}
    \label{E1}
    \bar{E}\left( \phi,s \right) =\frac{\varepsilon ^2}{2}\left( \mathcal{L}\phi ,\phi \right) +s.
\end{align}
We can rewrite the NAC equation \eqref{eqNAC} as:
\begin{subequations}
    \label{eq0}
    \begin{align}
        &\frac{\partial \phi}{\partial t}=- \varepsilon ^2\mathcal{L}\phi +g\left( \phi ,s \right) f(\phi)  \label{eq0-1},\\
        &\frac{ds}{dt}=-g\left( \phi ,s \right) \left( f\left( \phi \right) ,\frac{\partial \phi}{\partial t} \right) \label{eq0-2},        
        \end{align}  
    \end{subequations}
where
$$
g\left( \phi ,s \right) :=\frac{\exp \left\{ s \right\}}{\exp \left\{ E_1\left( \phi \right) \right\}}>0,\quad f(\phi) = -F'(\phi).
$$
\par
From now on, we assume that the initial value $\phi_{0}$ 
has sufficient regularity as required.
From the spatial discretization, there is a constant $h_0>0$ 
that depends on $\phi_{0},F,\varepsilon $, such that
\begin{align}
    \label{eqh0}
    E_h(\phi_{0}) \leq E(\phi_{0})+1,\quad \lVert \mathcal{L}_h\phi \rVert \le \lVert \mathcal{L}\phi \rVert _{L^2}+1,\quad \forall h\in (  0,h_0  ]  . 
\end{align} 
\par
\subsection{First-order sESAV1 scheme }
In this paper, $\phi^n$ stands for the fully discrete approximation of 
the exact solution $\phi(\mathbf{x},t)$ of the NAC equation \eqref{eqNAC}.
A semi-implicit first-order fully discrete sESAV (sESAV1) scheme for \eqref{eq0} reads as 
\begin{subequations}
    \label{eq11}
    \begin{align}
        &\delta _t\phi ^{n+1}=- \varepsilon ^2\mathcal{L}_h\phi ^{n+1}+g\left( \phi ^n,s^n \right) f\left( \phi ^n \right)  -\kappa g\left( \phi ^n,s^n \right) \left( \phi ^{n+1}-\phi ^n \right) ,\label{eq1-1}\\
        &\delta _ts^{n+1}=-g\left( \phi ^n,s^n \right) \left< f\left( \phi ^n \right) ,\delta _t\phi ^{n+1} \right> ,\label{eq1-2}
        \end{align}   
    \end{subequations}
where 
$$
g\left( \phi ^n,s^n \right) :=\frac{\exp \left\{ s^n \right\}}{\exp \left\{ E_{1h}\left( \phi ^n \right) \right\}}>0,\quad
\delta _t\phi ^{n+1}=\frac{\phi ^{n+1}-\phi ^n}{\varDelta t},
$$
and $E_{1h}$ is defined as $E_{1h}(\omega):=<F(\omega ),1>,\omega \in \mathcal{M}_h, $
$\kappa \ge 0$ is the stability constant, and for the specific value, refer to \cite{10}.
\par
The scheme \eqref{eq11} is started by $\phi^0=\phi_0$ and $s^0=E_{1h}(\phi^0)$.
We can rewrite the scheme \eqref{eq11} equivalently as follows:
\begin{subequations}
    \label{eq12}
    \begin{align}
        \left[ \left( \frac{1}{\tau}+\kappa g\left( \phi ^n,s^n \right) \right) I+\varepsilon ^2\mathcal{L}_h \right] \phi ^{n+1}&=\left( \frac{1}{\tau}+\kappa g\left( \phi ^n,s^n \right) \right) \phi ^n+g\left( \phi ^n,s^n \right) f\left( \phi ^n \right) ,\label{eq2-1}\\
        s^{n+1}&=s^n-g\left( \phi ^n,s^n \right) \left< f\left( \phi ^n \right) ,\phi ^{n+1}-\phi ^n \right> .\label{eq2-2}
        \end{align}   
    \end{subequations}
    \par
Obviously, since the matrix 
$\left( \frac{1}{\tau}+\kappa g\left( \phi ^n,s^n \right) \right) I+\varepsilon ^2\mathcal{L}_h$
is positive definite for any $\tau>0$, 
we know that the solution to the system \eqref{eq12} exists and is unique.
Furthermore, it can be seen that this numerical scheme \eqref{eq12} 
is very simple for the calculation of unknowns 
that is
$\phi^{n+1}$ can be solved linearly by equation \eqref{eq2-1}
and then $s^{n+1}$ can be solved directly by equation \eqref{eq2-2}.

\subsubsection{Energy dissipation law and MBP}
First we can write the spatically-discretized scheme of 
the modified nonlocal free energy \eqref{E1} as:
\begin{align}
    \label{eqDE}
    \bar{E}_h\left( \phi ^n,s^n \right) :=\frac{\varepsilon ^2}{2}\left< \mathcal{L}_h\phi ^n,\phi ^n \right> +s^n.   
\end{align}
\begin{theorem}{(Energy dissipation of sESAV1)}
    \label{t1}
    For any $\kappa \ge 0$ and $\tau > 0$,the sESAV1 scheme is 
    unconditionally energy stable in the sense that 
    $\bar{E}_h\left( \phi ^{n+1},s^{n+1} \right) \le \bar{E}_h\left( \phi ^n,s^n \right)$. 
\end{theorem}
\begin{proof}
    Taking the discrete inner product with \eqref{eq1-1} by $\phi^{n+1}-\phi^n$ yields
    \begin{equation}
        \label{eq1}
        \begin{aligned}            
            \left( \frac{1}{\tau}+\kappa g\left( \phi ^n,s^n \right) \right) \lVert \phi ^{n+1}-\phi ^n \rVert ^2
            =-\varepsilon ^2\left< \mathcal{L}_h\phi ^{n+1},\phi ^{n+1}-\phi ^n \right> +g\left( \phi ^n,s^n \right) \left< f\left( \phi ^n \right) ,\phi ^{n+1}-\phi ^n \right> .
        \end{aligned}
    \end{equation}
    Combining \eqref{eq1-2}, \eqref{eq1}, and the identity
    \begin{align*}
        2\left< \mathcal{L}_h\phi ^{n+1},\phi ^{n+1}-\phi ^n \right> 
        =\left< \mathcal{L}_h\left( \phi ^{n+1}-\phi ^n \right) ,\phi ^{n+1}-\phi ^n \right> +\left< \mathcal{L}_h\phi ^{n+1},\phi ^{n+1} \right> -\left< \mathcal{L}_h\phi ^n,\phi ^n \right> ,
    \end{align*}
    we obtain
    \begin{align*}
        &\bar{E}_h\left( \phi ^{n+1},s^{n+1} \right) -\bar{E}_h\left( \phi ^n,s^n \right) \\
        &=-\left( \frac{1}{\tau}+\kappa g\left( \phi ^n,s^n \right) \right) \lVert \phi ^{n+1}-\phi ^n \rVert ^2-\varepsilon ^2\left< \mathcal{L}_h\left( \phi ^{n+1}-\phi ^n \right) ,\phi ^{n+1}-\phi ^n \right> \\
        &\le 0,
    \end{align*}
    where $\kappa \ge 0$, $g(\phi^n,s^n)>0$ and the operator $\mathcal{L}_h$ is positive semi-definite.
\end{proof}
\par
Theorem \ref{t1} shows that the scheme \eqref{eq11} satisfies energy dissipation with 
respect to the modified energy $\bar{E}_h\left( \phi ^n,s^n \right)$  
rather than the original energy $E_h(\phi^n)$.
\begin{theorem}{(MBP of sESAV1)}
    \label{t2}
    If $\kappa \geq \lVert f' \rVert _{C\left[ -\beta ,\beta \right]}$,
    the sESAV1 \eqref{eq11} preserves the MBP for $\{\phi^n\}$,i.e.,
    the fully discrete scheme of \eqref{MBP} holds:
    \begin{align}
        \label{DMBP}
        \lVert \phi _{0} \rVert _{\infty}\le \beta \Rightarrow \lVert \phi ^n \rVert _{\infty}\le \beta .
    \end{align}
\end{theorem}
\begin{proof}
    We can prove this theorem by mathematical induction.
    Suppose $(\phi^n,s^n)$ is given and 
    $\lVert \phi ^n \rVert _{\infty}\le \beta$ for some $n$.
    According to \eqref{eq2-1}, we obtain
    $$
    \phi ^{n+1}=\left[ \left( \frac{1}{\tau}+\kappa g\left( \phi ^n,s^n \right) \right) I+\varepsilon ^2\mathcal{L}_h \right] ^{-1}\left[ \left( \frac{1}{\tau}+\kappa g\left( \phi ^n,s^n \right) \right) \phi ^n+g\left( \phi ^n,s^n \right) f\left( \phi ^n \right) \right] 
    .
    $$
    Since $g(\phi^n,s^n)>0$, by Lemma \ref{l1}, we get
    $$
    \left\lVert  \left[ \left( \frac{1}{\tau}+\kappa g\left( \phi ^n,s^n \right) \right) I+\varepsilon ^2\mathcal{L}_h \right]  ^{-1} \right\rVert _\infty\le \left( \frac{1}{\tau}+\kappa g\left( \phi ^n,s^n \right) \right) ^{-1}.
    $$
    Combining $\kappa \ge \lVert f' \rVert _{C\left[ -\beta ,\beta \right]}$,
    $\lVert \phi ^n \rVert _{\infty}\le \beta$,
    and Lemma \ref{l2}, it has
    $$
    \left\lVert \frac{1}{\tau}\phi ^n+g\left( \phi ^n,s^n \right) \left( f\left( \phi ^n \right) +\kappa \phi ^n \right) \right\rVert _{\infty}\le \left( \frac{1}{\tau}+\kappa g\left( \phi ^n,s^n \right) \right) \beta .
    $$
    So, we have 
    $$
    \lVert \phi ^{n+1} \rVert _{\infty}\le \left( \frac{1}{\tau}+\kappa g\left( \phi ^n,s^n \right) \right) ^{-1}\left( \frac{1}{\tau}+\kappa g\left( \phi ^n,s^n \right) \right) \beta =\beta .
    $$
    Therefore, by induction, $\lVert \phi ^n \rVert _{\infty}\le \beta$, for all $n$.
\end{proof}

\subsection{Second-order sESAV2 scheme}
Our second-order fully discrete sESAV (sESAV2) is given below: 
\begin{subequations}
    \label{eq21}
    \begin{align}
        &\delta _t\phi ^{n+1}=-\varepsilon ^2\mathcal{L}_h\phi ^{n+\frac{1}{2}}+g( \bar{\phi}^{n+\frac{1}{2}},\bar{s}^{n+\frac{1}{2}} ) f( \bar{\phi}^{n+\frac{1}{2}} ) -\kappa g( \bar{\phi}^{n+\frac{1}{2}},\bar{s}^{n+\frac{1}{2}} ) ( \phi ^{n+\frac{1}{2}}-\bar{\phi}^{n+\frac{1}{2}} ) 
        ,\label{eqCN-1}\\
        &\delta _ts^{n+1}=-g( \bar{\phi}^{n+\frac{1}{2}},\bar{s}^{n+\frac{1}{2}} ) \left< f( \bar{\phi}^{n+\frac{1}{2}} ) ,\delta _t\phi ^{n+1} \right> +\kappa g( \bar{\phi}^{n+\frac{1}{2}},\bar{s}^{n+\frac{1}{2}} ) \left< \phi ^{n+\frac{1}{2}}-\bar{\phi}^{n+\frac{1}{2}},\delta _t\phi ^{n+1} \right> ,
        \label{eqCN-2}
        \end{align}   
    \end{subequations}
where
$$
g( \bar{\phi}^{n+\frac{1}{2}},\bar{s}^{n+\frac{1}{2}} ) :=\frac{\exp \left\{ \bar{s}^{n+\frac{1}{2}} \right\}}{\exp \left\{ E_1( \bar{\phi}^{n+\frac{1}{2}} ) \right\}}>0,
\quad
\phi ^{n+\frac{1}{2}}=\frac{\phi ^{n+1}+\phi ^n}{2},
$$
and $( \bar{\phi}^{n+\frac{1}{2}},\bar{s}^{n+\frac{1}{2}} )$ are arbitrary explicit
$O(\tau^2)$ appropriation for $( \phi^{n+\frac{1}{2}},s^{n+\frac{1}{2}})$ respectively.
For instance, we may choose the extrapolation
\begin{subequations}
    \label{eqss}
    \begin{align}   
        \bar{\phi} ^{n+\frac{1}{2}}&=\frac{3}{2}\phi ^n-\frac{1}{2}\phi ^{n-1},\label{ss1}\\
        \bar{s}^{n+\frac{1}{2}}&=\frac{3}{2}s^n-\frac{1}{2}s^{n-1},n\ge 1,
    \end{align}
\end{subequations}
or the semi-implicit scheme
\begin{subequations}
    \label{eqs}
    \begin{align}
        \frac{\bar{\phi}^{n+\frac{1}{2}}-\phi ^n}{\tau /2}&=-\varepsilon ^2\mathcal{L}_h\bar{\phi}^{n+\frac{1}{2}}+g\left( \phi ^n,s^n \right) f\left( \phi ^n \right) -\kappa g\left( \phi ^n,s^n \right) ( \bar{\phi}^{n+\frac{1}{2}}-\phi ^n ) ,\label{s1} \\
        \bar{s}^{n+\frac{1}{2}}-s^n&=-g\left( \phi ^n,s^n \right) \left< f\left( \phi ^n \right) ,\bar{\phi}^{n+\frac{1}{2}}-\phi ^n \right>,n\ge 0.
        \end{align}   
    \end{subequations}
\par
The system \eqref{eq21} is started by $\phi^0=\phi_0$, $s^0=E_{1h}(\phi^0)$.
We can rewrite the system \eqref{eq21} in the following form:
\begin{subequations}
    \label{eq22}
    \begin{align}
        &D_1\phi ^{n+1}=D_2\phi ^n+2g( \bar{\phi}^{n+\frac{1}{2}},\bar{s}^{n+\frac{1}{2}} ) \left[ f( \bar{\phi}^{n+\frac{1}{2}} ) +\kappa \bar{\phi}^{n+\frac{1}{2}} \right],
        \label{eqCN2-1}\\
        &s^{n+1}=s^n-g( \bar{\phi}^{n+\frac{1}{2}},\bar{s}^{n+\frac{1}{2}} ) \left< f( \bar{\phi}^{n+\frac{1}{2}} ) -\kappa ( \phi ^{n+\frac{1}{2}}-\bar{\phi}^{n+\frac{1}{2}} ) ,\phi ^{n+1}-\phi ^n \right>,
        \label{eqCN2-2}\\
        &D_1=\left( \frac{2}{\tau}+\kappa g( \bar{\phi}^{n+\frac{1}{2}},\bar{s}^{n+\frac{1}{2}} ) \right) I+\varepsilon ^2\mathcal{L}_h,\notag\\
        &D_2=\left( \frac{2}{\tau}-\kappa g( \bar{\phi}^{n+\frac{1}{2}},\bar{s}^{n+\frac{1}{2}} ) \right) I-\varepsilon ^2\mathcal{L}_h\notag.
        \end{align}   
    \end{subequations}
    \par
It is easy to see that since $D_1$ is positive definite for any $\tau>0$. 
So the system \eqref{eq22} is linear and uniquely solvable,
and the unknowns $\phi^{n+1}$, $s^{n+1}$
are computed similarly to the sESAV1 scheme.

\subsubsection{Energy dissipation law and MBP}
\begin{theorem}{(Energy dissipation of sESAV2)}
    \label{t3}
    For any $\kappa \ge 0$ and $\tau > 0$,the sESAV2 scheme is unconditionally energy stable in the sense that 
    $\bar{E}_h\left( \phi ^{n+1},s^{n+1} \right) \le \bar{E}_h\left( \phi ^n,s^n \right)$,
    where $\bar{E}_h\left( \phi ^n,s^n \right)$ is given by \eqref{eqDE}.
\end{theorem}
\begin{proof}
    Taking the discrete inner product with \eqref{eqCN-1} by $\phi^{n+1}-\phi^n$ yields
    \begin{equation}
        \label{eq2}
        \begin{aligned}    
            \frac{1}{\tau}\lVert \phi ^{n+1}-\phi ^n \rVert ^2=&-\frac{\varepsilon ^2}{2} \left< \mathcal{L}_h\phi ^{n+1},\phi ^{n+1} \right> +\frac{\varepsilon ^2}{2}\left< \mathcal{L}_h\phi ^n,\phi ^n \right> \\
            &-g( \bar{\phi}^{n+\frac{1}{2}},\bar{s}^{n+\frac{1}{2}} ) \left< f( \bar{\phi}^{n+\frac{1}{2}} ) ,\phi ^{n+1}-\phi ^n \right> \\
            &-\kappa g( \bar{\phi}^{n+\frac{1}{2}},\bar{s}^{n+\frac{1}{2}} ) \left< \phi ^{n+\frac{1}{2}}-\bar{\phi}^{n+\frac{1}{2}},\phi ^{n+1}-\phi ^n \right> 
        \end{aligned}
    \end{equation}
    Combining \eqref{eqCN-2},\eqref{eq2}, we obtain
    \begin{align*}
        \bar{E}_h\left( \phi ^{n+1} \right) -\bar{E}_h\left( \phi ^n \right) =-\frac{1}{\tau}\lVert \phi ^{n+1}-\phi ^n \rVert ^2\le 0,
    \end{align*}
    then we complete the proof. 
\end{proof}

Note that the coefficient $g(\phi^n,s^n)$ may vary step-by-step, which is different
from the continuous case that $g(\phi,s)\equiv 1$ exactly, but by \cite{33}, 
we know that it can be bounded 
by some constant, which is illustrated in the following:
\begin{corollary}
    \label{c1}
    if $h \le h_0, \kappa \ge \lVert f' \rVert _{C\left[ -\beta ,\beta \right]}$ and $\lVert \phi _0 \rVert _\infty \le \beta$,
    then there exists a constant $C^*=C^*(\phi_0,C_*)$, such that $0<g\left( \bar{\phi}^{n+\frac{1}{2}},\bar{s}^{n+\frac{1}{2}} \right) \le C^*$
    for all $n$.
\end{corollary}
\begin{theorem}{(MBP of sESAV2)}
    If $h \le h_0,\kappa \geq \lVert f' \rVert _{C\left[ -\beta ,\beta \right]}$, and
    \begin{align}
        \label{ss}
        \tau \le \left( \frac{\kappa G^*}{2}+\frac{\varepsilon ^2}{h^2} \right) ^{-1},
    \end{align}
    then the sESAV2 scheme \eqref{eq21}
    preserves the MBP for $\{\phi^n\}$, i.e., \eqref{DMBP} is valid.
\end{theorem}
\begin{proof}
    From \eqref{eqCN2-1}, we get
    $$
    \phi ^{n+1}=D_1^{-1}\left[ D_2\phi ^n+2g( \bar{\phi}^{n+\frac{1}{2}},\bar{s}^{n+\frac{1}{2}} ) \left( f( \bar{\phi} ^{n+\frac{1}{2}} ) +\kappa \bar{\phi}^{n+\frac{1}{2}} \right) \right] 
    .
    $$
    Suppose $(\phi^n,s^n)$ is given and $\lVert \phi ^n \rVert _{\infty}\le \beta $ for 
    some $n$. By Corollary \ref{c1} and condition \eqref{ss}, 
    we have
    \begin{align*}
        \frac{2}{\tau}-\kappa g( \bar{\phi}^{n+\frac{1}{2}},\bar{s}^{n+\frac{1}{2}}) \ge \frac{2\varepsilon ^2}{h^2}>0.
    \end{align*} 
    Using Lemma \ref{l1}, so
    $$
    \lVert D_1^{-1} \rVert _{\infty}=\lVert \left[ \left( \frac{2}{\tau}+\kappa g( \bar{\phi}^{n+\frac{1}{2}},\bar{s}^{n+\frac{1}{2}} ) \right) I+\varepsilon ^2\mathcal{L}_h \right]^{-1} \rVert _{\infty}\le \left( \frac{2}{\tau}+\kappa g( \bar{\phi}^{n+\frac{1}{2}},\bar{s}^{n+\frac{1}{2}} ) \right) ^{-1}.
    $$ 
    According to the definition of the matrix $\infty-$norm, we have 
    $$
    \lVert D_2 \rVert _{\infty}=\lVert \left[ \left( \frac{2}{\tau}-\kappa g( \bar{\phi}^{n+\frac{1}{2}},\bar{s}^{n+\frac{1}{2}} ) \right) I-\varepsilon ^2\mathcal{L}_h \right] \rVert _{\infty}=\frac{2}{\tau}-\kappa g( \bar{\phi}^{n+\frac{1}{2}},\bar{s}^{n+\frac{1}{2}} ) .
    $$
    Since $\kappa \ge \lVert f' \rVert _{C\left[ -\beta ,\beta \right]}$ 
    and $\lVert \bar{\phi}^{n+\frac{1}{2}} \rVert _{\infty}\le \beta $ 
    which can be obtained by the Theorem \ref{t1} and the expression \eqref{s1} or \eqref{ss1},
    according to Lemma \ref{l2}, we have 
    $$
    \lVert f( \bar{\phi}^{n+\frac{1}{2}} ) +\kappa \bar{\phi}^{n+\frac{1}{2}} \rVert _{\infty}\le \kappa \beta .
    $$
    Therefore, we obtain from \eqref{eqCN2-1} that
    $$
    \lVert \phi ^{n+1} \rVert _{\infty}\le \left[ \frac{2}{\tau}+\kappa g( \bar{\phi}^{n+\frac{1}{2}},\bar{s}^{n+\frac{1}{2}} ) \right] ^{-1}\left[ \left( \frac{2}{\tau}-\kappa g( \bar{\phi}^{n+\frac{1}{2}},\bar{s}^{n+\frac{1}{2}} ) \right) \beta +2g( \bar{\phi}^{n+\frac{1}{2}},\bar{s}^{n+\frac{1}{2}} ) \kappa \beta \right] =\beta .
    $$
    By induction, we have $\lVert \phi ^n \rVert _{\infty}\le \beta $ for all $n$.
    \end{proof}

Owing to $g( \bar{\phi}^{n+\frac{1}{2}},\bar{s}^{n+\frac{1}{2}} ) \approx 1$ 
in the numerical calculation, the time step only needs to meet 
$\tau \le \left( \frac{\kappa}{2}+\frac{\varepsilon ^2}{h^2} \right) ^{-1}$ 
in order to preserve the MBP.

\section{Numerical experiments}
This section is devoted to carry out various numerical experiments to 
test and compare the effectiveness and efficiency of the proposed 
sESAV1 \eqref{eq11} and sESAV2 \eqref{eq22} schemes, 
particularly on the numerical energy stability and MBP. 
We consider the NAC equationin \eqref{eqNAC} in two-dimensional 
spatial domain $\varOmega =(-1, 1) \times (-1, 1) $ 
equipped with periodic boundary conditions, 
so that the schemes can be solved efficiently by the fast solver.
We define the Gaussian kernel as follows \cite{5}
\begin{align}
    \label{J}
    J_{\delta}=\frac{4}{\pi ^{\frac{d}{2}}\delta ^{d+2}}e^{-\frac{|\mathbf{x|}^2}{\delta ^2}},\quad \mathbf{x}\in \varOmega \subseteq \mathbb{R}^2,\quad \delta >0,
\end{align}
and it satisfies that for any $\psi \in C^{\infty}\left( \varOmega \right) ,\mathbf{x}\in \varOmega ,\mathcal{L}_{\delta}\varphi \left( \mathbf{x} \right) \rightarrow -\varDelta \psi \left( \mathbf{x} \right) $ as $\delta \rightarrow 0$.
\par 
We take two types of commonly-used nonlinear functions $f(\phi)$. 
One is given by
\begin{align}
    f(\phi) = -F'(\phi) = \phi - \phi^3 \label{4-1}
\end{align}
with $F$ being the double-well potential
\begin{align*}
    F(\phi) = \frac{1}{4}(\phi^2-1)^2.
\end{align*}
In this case, $\beta = 1$ and $\lVert f' \rVert _{C\left[ -\beta ,\beta \right]}=2$.\par
The other one is the Flory-Huggins potential
$$
F\left( \phi \right) =\frac{\theta}{2}\left[ \left( 1+\phi \right) \ln \left( 1+\phi \right) +\left( 1-\phi \right) \ln \left( 1-\phi \right) \right] -\frac{\theta _c}{2}\phi ^2,
$$
which gives
\begin{align}
    \label{4-2}
    f\left( \phi \right) =-F'\left( \phi \right) =\frac{\theta}{2}\ln \frac{\left( 1-\phi \right)}{\left( 1+\phi \right)}+\theta _c\phi ,
\end{align}
where $\theta _c >\theta  > 0$, so we set $\theta  = 0.8$, $\theta _c = 1.6$.
Then the positive root of $f(\rho ) = 0 $ gives $\beta \approx 0.9575$, 
and $\lVert f' \rVert _{C\left[ -\beta ,\beta \right]}\approx 8.02$.
A point to note here is that when the absolute difference between 
the modified discrete energies 
$\left| \bar{E}_h\left( \phi ^n \right) -\bar{E}_h\left( \phi ^{n-1} \right) \right|,n\ge 1$ 
at two consecutive times is less than the error value $1e-8$, 
we consider that the phase transition 
process of this experiment has reached a steady state.\par
\textbf{Example 1} (Convergence in time) 
We consider the NAC equation equipped 
with the Gaussian kernel $J_\delta $ \eqref{J},
and the related parameters are $ T=1, N=2^7 ,
\varepsilon  =0.05, \delta = 0.05$.
In addition, we set $\kappa=2, \beta=1 $ with 
the double-well potential \eqref{4-1};
$\kappa=8.02, \beta=0.9575$ with the Flory-Huggins potential \eqref{4-2}, 
and the initial value is
$$
\phi _0\left( x,y \right) =0.5\cos \left( \pi x \right) \cos \left( \pi y \right). 
$$
It should be noted that we do not have exact solutions for comparison, 
so we calculate the reference solution $\phi_{ref}$ with $\tau=T/2^{15}$ 
by the second-order sESAV2 scheme for the error calculation.
In order to observe the temporal convergence rate,
by taking the linear refinement path $\tau =2^{-k}, k=4,5,...,10$, 
we use discrete $L^2$ norm between the 
difference ($\lVert \phi ^n-\phi _{ref} \rVert,n\ge 0$) 
to calculate the convergence rate for the sESAV1 and sESAV2 schemes.
Table \ref{T1} and Table \ref{T2} show the 
discrete $L^2$ norm error and temporal 
convergence accuracy for different $f(\phi)$ respectively. 
We can observe that the first-order sESAV1 scheme and 
the second-order sESAV2 scheme converge to 
the reference solution asymptotically with the expected accuracy.
\begin{table}[htp]
    \centering    
    \renewcommand\arraystretch{1.5}  
    \caption{the $L^2$ norm errors and temporal convergence rates of the sESAV1 and SESAV2 for the double-well potential case.}
    \label{T1}
    \begin{tabular}{lllllll}
        \toprule
        \multirow{2}{*}{$\tau$}   & \multicolumn{2}{l}{sESAV1} & \multicolumn{2}{l}{sESAV2}\\
        \cmidrule(lr){2-3}  \cmidrule(lr){4-5}
        & $Error$ & $Rate$ & $Error$ & $Rate$\\
        \midrule
        $T/2^5$  & 1.8938E-02 & - & 6.4040E-04 & - \\
        $T/2^6$  & 9.7621E-03 & 0.9560 & 1.6569E-04 & 1.9505 \\
        $T/2^7$  & 4.9574E-03 & 0.9776 & 4.2154E-05 & 1.9747 \\
        $T/2^8$  & 2.4982E-03 & 0.9887 & 1.0632E-05 & 1.9873 \\
        $T/2^9$  & 1.2540E-03 & 0.9943 & 2.6693E-06 & 1.9938 \\
        $T/2^{10}$ & 6.2824E-04 & 0.9972 & 6.6832E-07 & 1.9978 \\
        \bottomrule
        \end{tabular}
\end{table}

\begin{table}[htp]
    \centering    
    \renewcommand\arraystretch{1.5}  
    \caption{the $L^2$ norm errors and temporal convergence rates of the sESAV1 and SESAV2 for the Flory-Huggins potential case.}
    \label{T2}
    \begin{tabular}{lllllll}
        \toprule
        \multirow{2}{*}{$\tau$}   & \multicolumn{2}{l}{sESAV1} & \multicolumn{2}{l}{sESAV2}\\
        \cmidrule(lr){2-3}  \cmidrule(lr){4-5}
        & $Error$ & $Rate$ & $Error$ & $Rate$\\
        \midrule
        $T/2^5$  & 9.5718E-02 & - & 6.7652E-03 & - \\
        $T/2^6$  & 5.3575E-02 & 0.8372 & 1.8948E-03 & 1.8361 \\
        $T/2^7$  & 2.8455E-02 & 0.9129 & 5.0307E-04 & 1.9132 \\
        $T/2^8$  & 1.4679E-02 & 0.9549 & 1.2972E-04 & 1.9553  \\
        $T/2^9$ & 7.4568E-03 & 0.9771  & 3.2940E-05  & 1.9775  \\
        $T/2^{10}$ & 3.7584E-03 & 0.9885 & 8.2947E-06  & 1.9896 \\
        \bottomrule
        \end{tabular}
\end{table}

\textbf{Example 2} (Unconditional energy stable and MBP)
In the following numerical experiments, we mainly analyze and 
discuss the numerical energy stability and MBP for the NAC equation.
Here we solve a benchmark problem \cite{50} for the NAC equation 
by using the fast solver.
The related parameters are $N=2^7 ,\tau =0.01, \varepsilon  =0.02, \delta = 0.02$.
We consider $\beta=1,\kappa=2$ with the the double-well potential \eqref{4-1}, 
and $\beta=0.9575,\kappa=8.02$ with the the Flory-Huggins potential \eqref{4-2}.
In order to verify the energy stability of the numerical simulation, 
we choose to observe and compare the evolution of
the original discrete energy $E_h(\phi^n)$
and the modified discrete energy $\bar{E}_h(\phi^n,s^n)$.
The initial condition is
$$
\phi _0\left( x,y \right) =\tan\text{h}\left( \frac{R_1-\sqrt{\left( x-a \right) ^2+\left( y-b \right) ^2}}{\sqrt{2}\varepsilon} \right) -\tan\text{h}\left( \frac{R_2-\sqrt{\left( x-a \right) ^2+\left( y-b \right) ^2}}{\sqrt{2}\varepsilon} \right) -1,
$$
where $R_1=0.8,R_2=0.6, (a,b)=(0,0).$
We use the sESAV2 scheme to simulate the long-time phase field evolution process.
For the double-well potential case,
the snapshots of the phase structure captured at $t=5,200,400,500,600,800$ are 
shown in Figure \ref{F2_1}, and the steady state $\phi \equiv -1 $ at about $t=900$.
In Figure \ref{F2_1}, we observe 
the evolution of concentric circles under the action of their mean curvature.
The smaller circle has greater curvature and contracts faster than the larger one. 
When the smaller circle disappears, the larger circle will continue to 
contract until it finally disappears.
The above situation is consistent with the fact that the NAC equation 
does not preserve the conservation of mass.
Figure \ref{F2_2} shows the evolution of the numerical modified energy
and original energy, and it is clear from the figures that the two 
energy curves decreases monotonically with time.
As can be seen from the top row pictures of Figure \ref{F2_3},
the supremum norm of phase variable $\phi^n,n\ge 0 $
is always less than $1$ in the evolution process of a long time.
From the bottom figures, it can be observed that the minimum value of 
the phase variable is no less than $-1$ and the maximum value is 
no more than $1$ in the evolution process of a long time,
which means that the MBP is preserved.
For the Flory-Huggins potential case, 
the snapshots of the phase structure captured at $t=3,15,50,100,200,350$ are 
shown in Figure \ref{F2_4}, 
and the steady state $\phi \equiv -1 $ arises at about $t=910$.
Furthermore, the phase transition is similar to the double-well potential case.
We can observe that both the original energy and 
the modified energy decrease with time in Figure \ref{F2_5}. 
For MBP, we can observe from Figure \ref{F2_6} that 
the supremum norm of phase variable $\phi^n,n\ge 0 $
is always less than $0.9575$.

\begin{figure}[htp]
    \centering
    \subfloat[t=5]{\includegraphics[width=0.33\textwidth,height=0.33\textwidth]{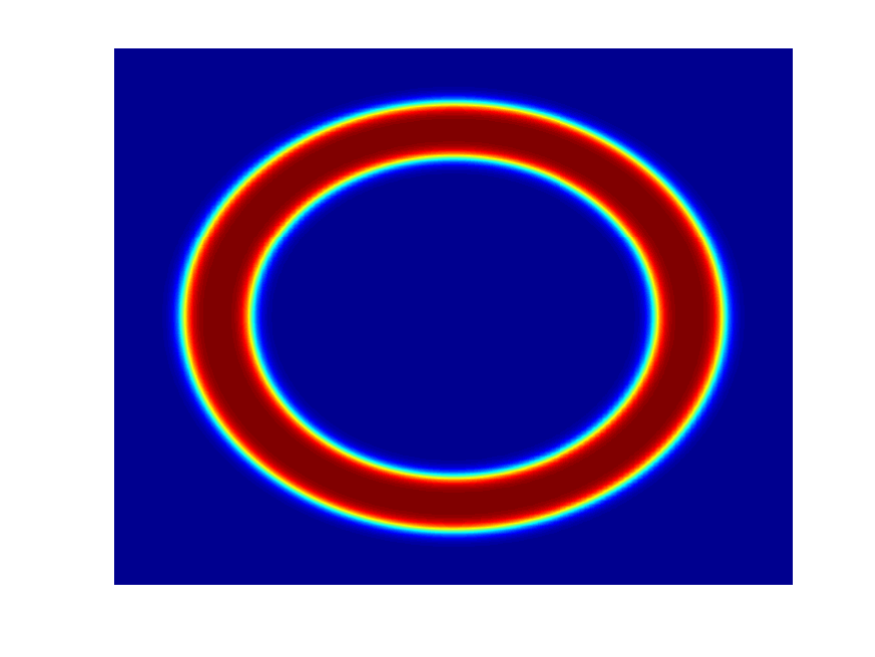}}
    \hfill
    \subfloat[t=200]{\includegraphics[width=0.33\textwidth,height=0.33\textwidth]{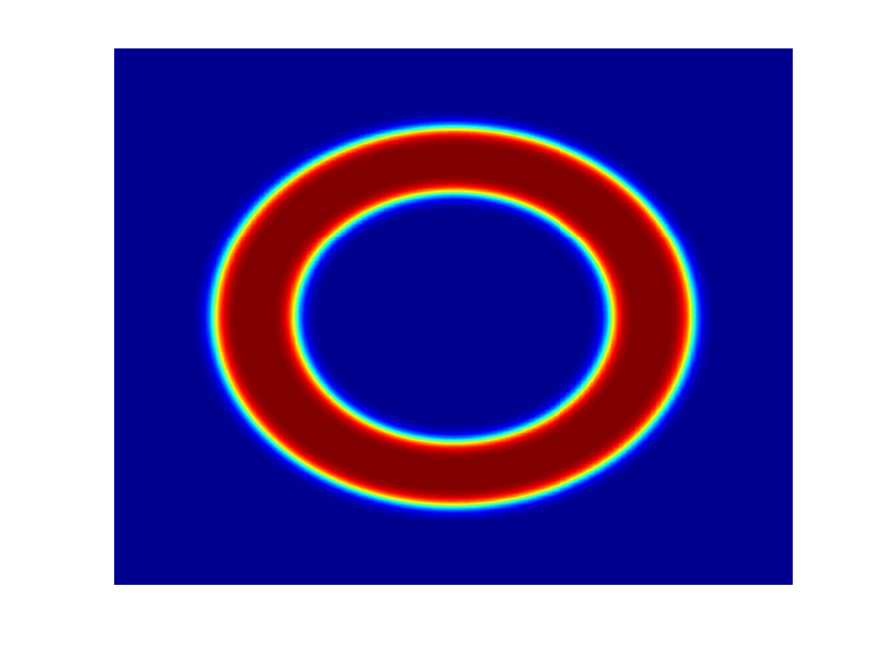}}
    \hfill
    \subfloat[t=400]{\includegraphics[width=0.33\textwidth,height=0.33\textwidth]{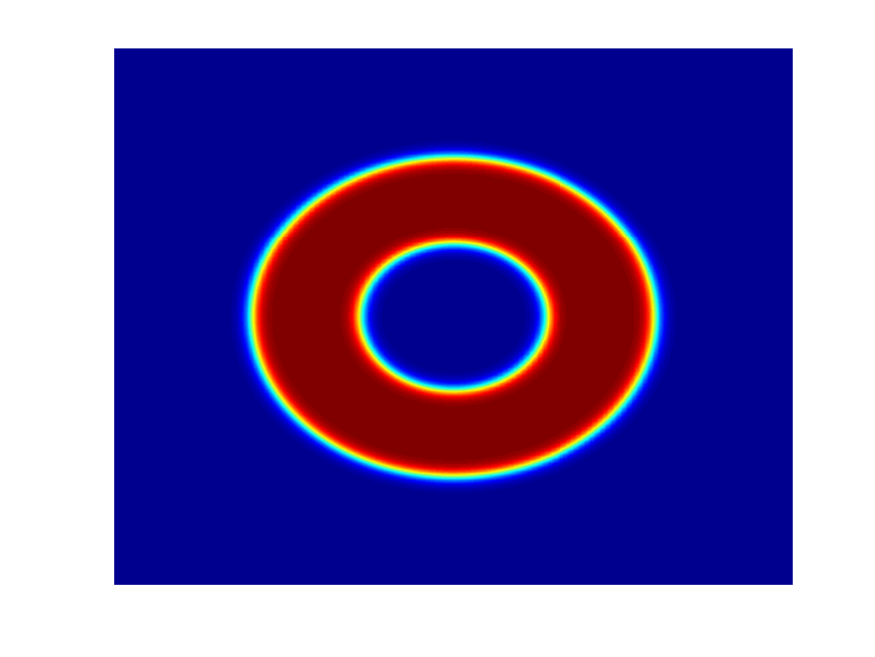}}
    \hfill
    \\
    \subfloat[t=500]{\includegraphics[width=0.33\textwidth,height=0.33\textwidth]{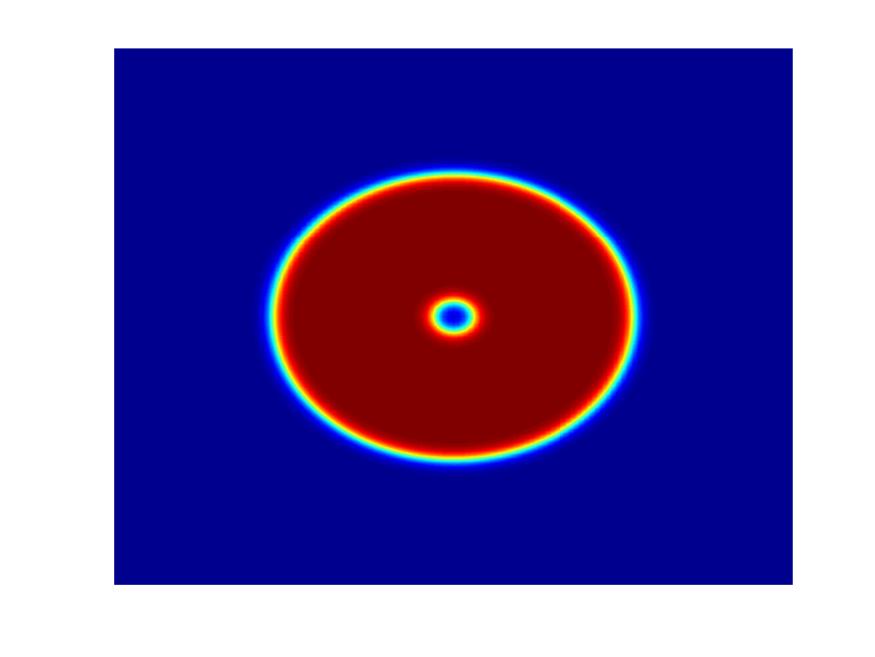}}
    \hfill
    \subfloat[t=600]{\includegraphics[width=0.33\textwidth,height=0.33\textwidth]{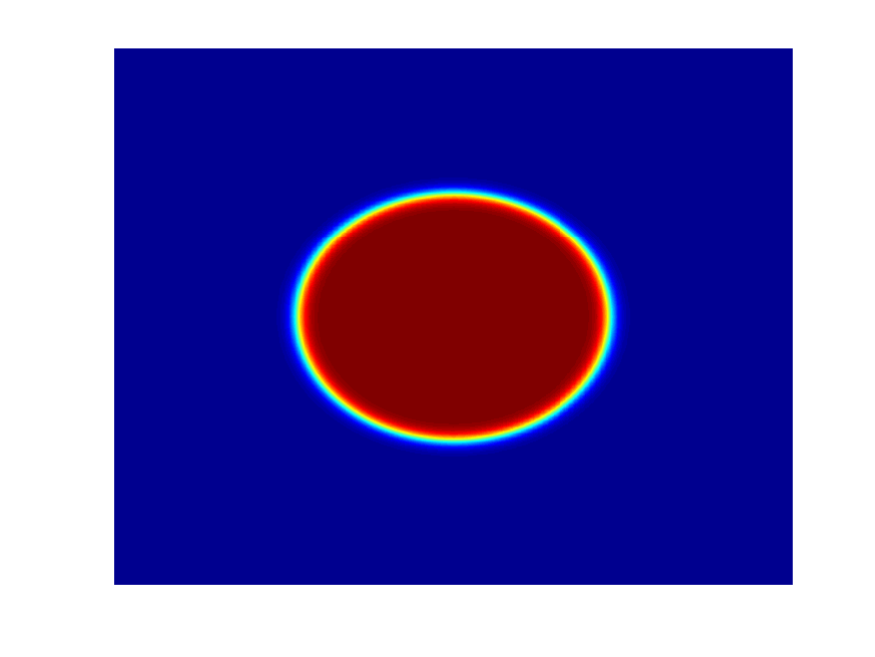}}
    \hfill
    \subfloat[t=800]{\includegraphics[width=0.33\textwidth,height=0.33\textwidth]{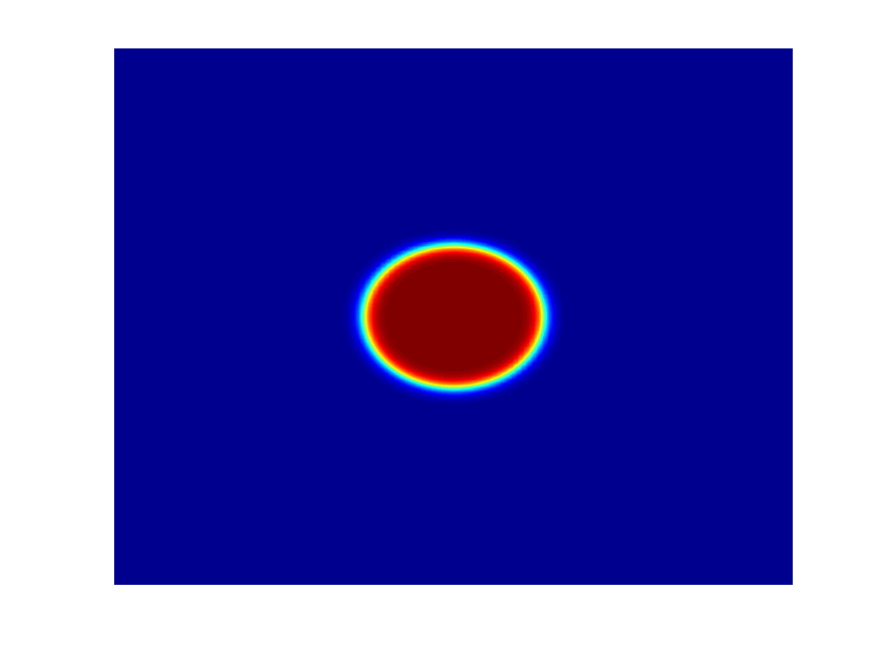}}
    \hfill
    \caption{Snapshots of the phase variable $\phi$ 
    are taken at $t=5,200,400,500,600,800$ with $\kappa =2,\beta=1$ 
    for the double-well potential case in \textbf{Example 2}}
    \label{F2_1}
    \hfill
\end{figure}

\begin{figure}[htp]
    \centering
    \subfloat{\includegraphics[scale=0.35]{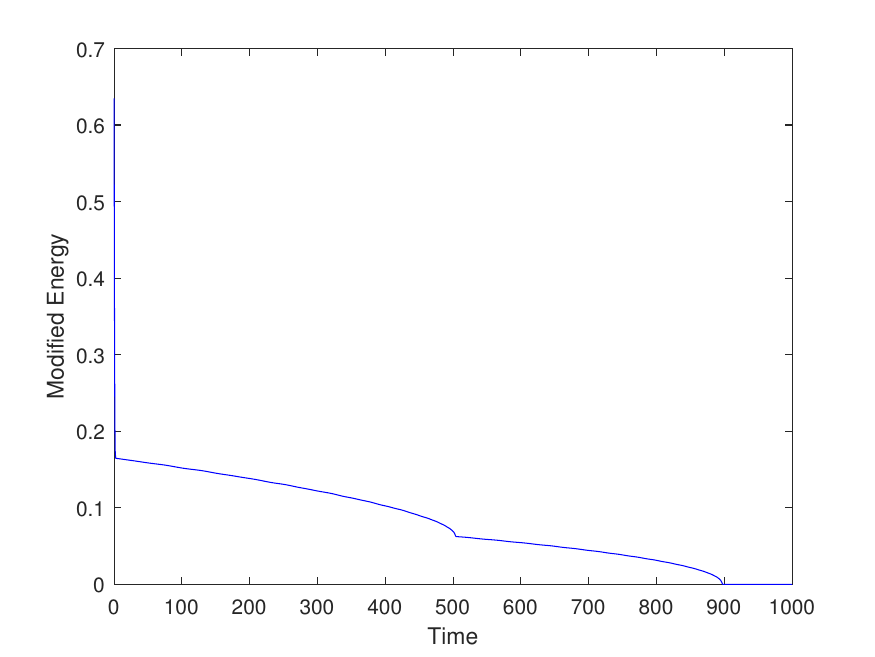}}
    \hfill
    \subfloat{\includegraphics[scale=0.35]{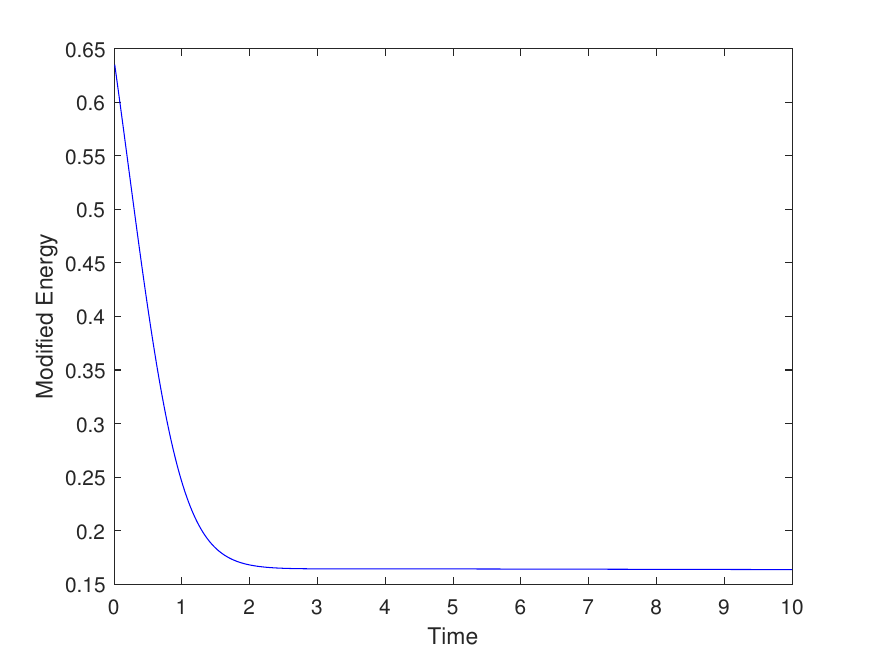}}
    \hfill
    \subfloat{\includegraphics[scale=0.35]{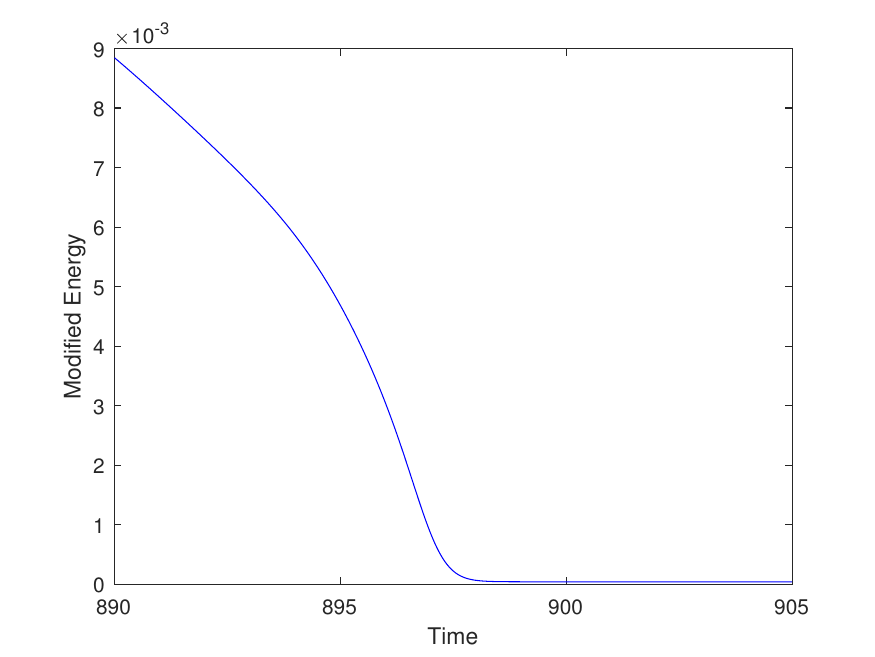}}
    \hfill
    \\
    \subfloat{\includegraphics[scale=0.35]{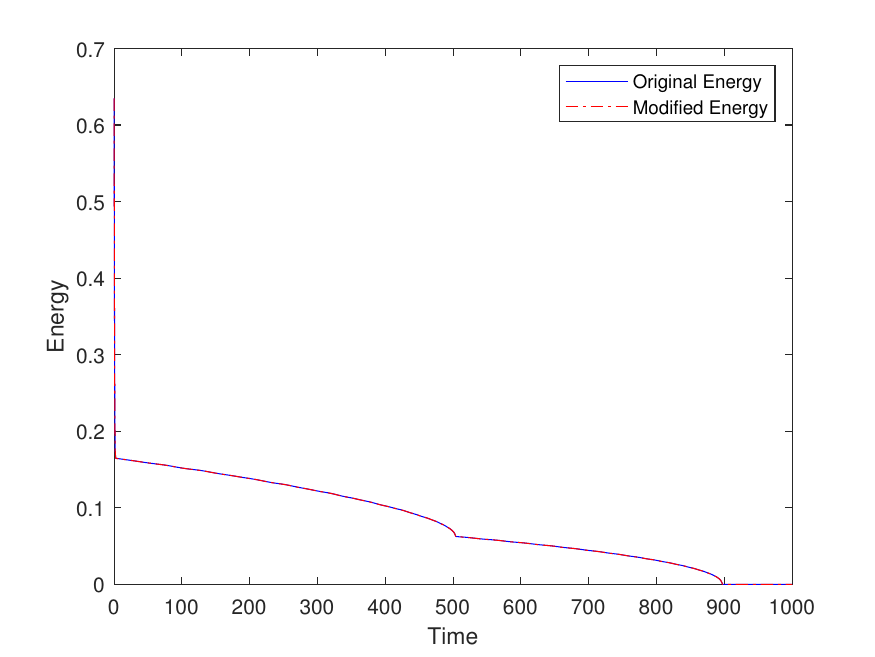}}
    \hfill
    \subfloat{\includegraphics[scale=0.35]{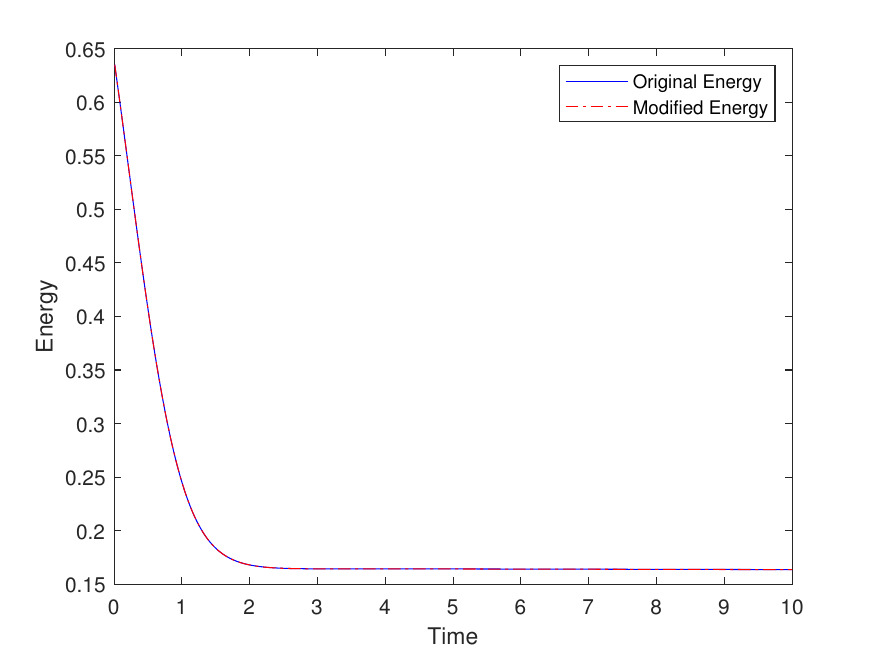}}
    \hfill
    \subfloat{\includegraphics[scale=0.35]{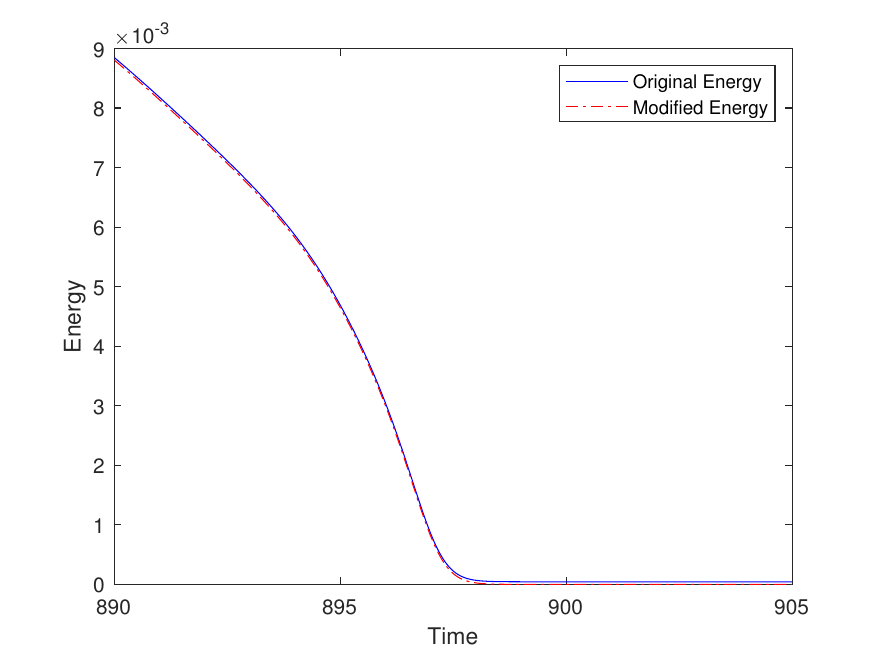}}
    \hfill
    \caption{Evolutions of the modified and original free 
    energy functional for the double-well potential case in \textbf{Example 2}}
    \label{F2_2}
    \hfill
\end{figure}

\begin{figure}[htp]
    \centering
    \subfloat{\includegraphics[scale=0.52]{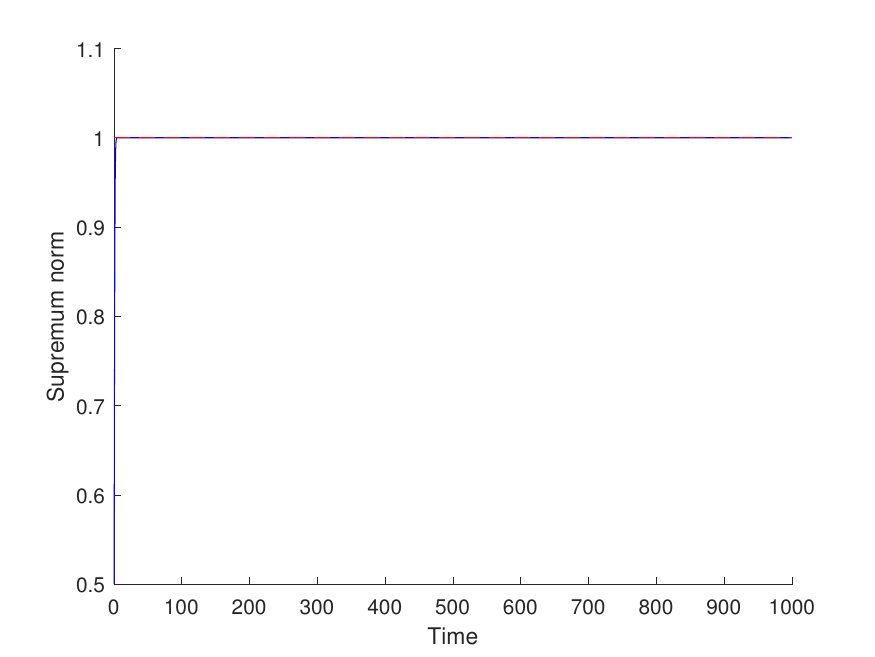}}
    \hfill
    \subfloat{\includegraphics[scale=0.52]{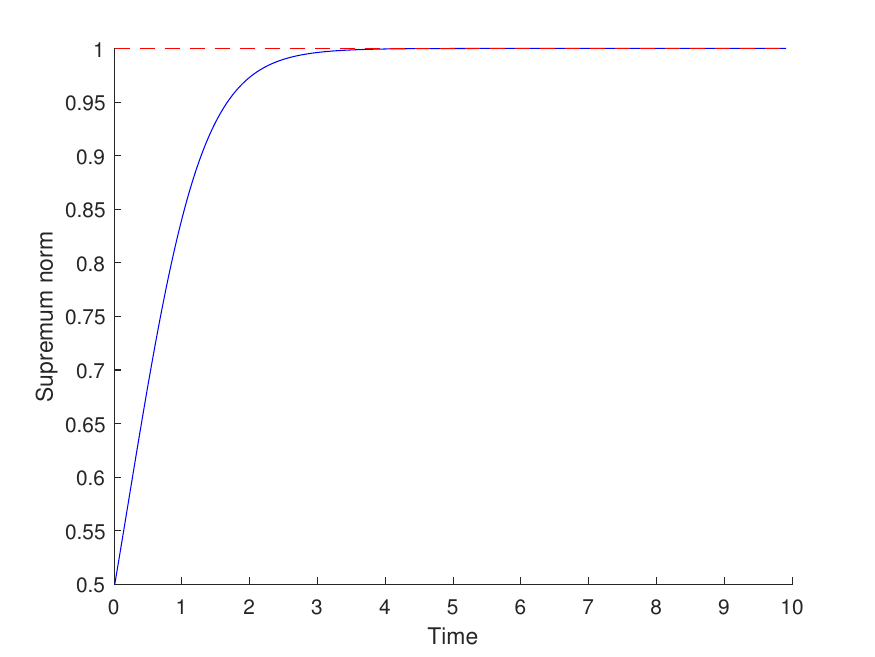}}
    \hfill
    \\
    \subfloat{\includegraphics[scale=0.52]{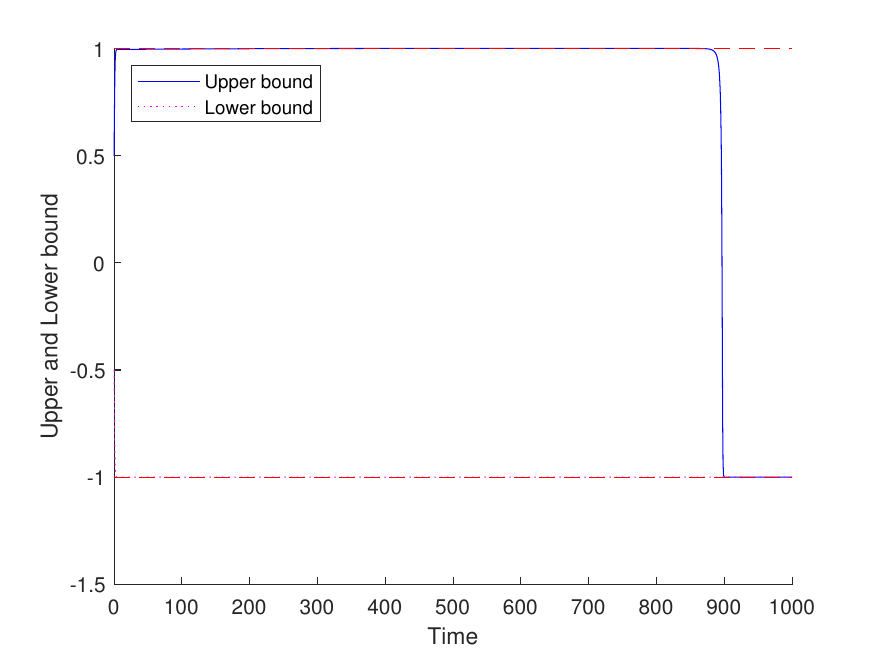}}
    \hfill
    \subfloat{\includegraphics[scale=0.52]{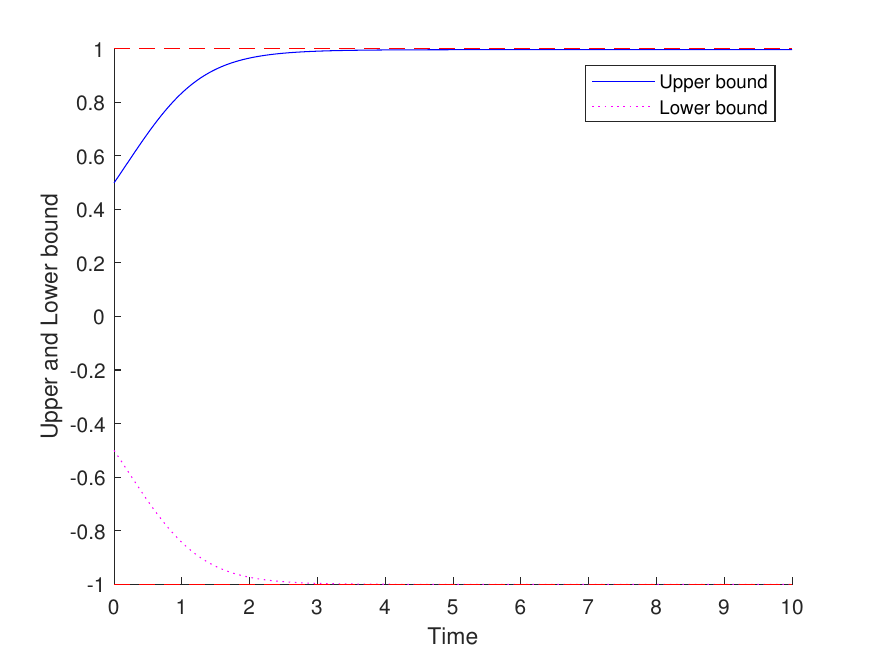}}
    \caption{Evolutions of the supremum norm (top row), 
    the upper and lower bound (bottom row) of the phase variable $\phi$
    for the double-well potential in \textbf{Example 2}}
    \label{F2_3}
    \hfill
\end{figure}

\begin{figure}[htp]
    \centering
    \subfloat[t=5]{\includegraphics[width=0.33\textwidth,height=0.33\textwidth]{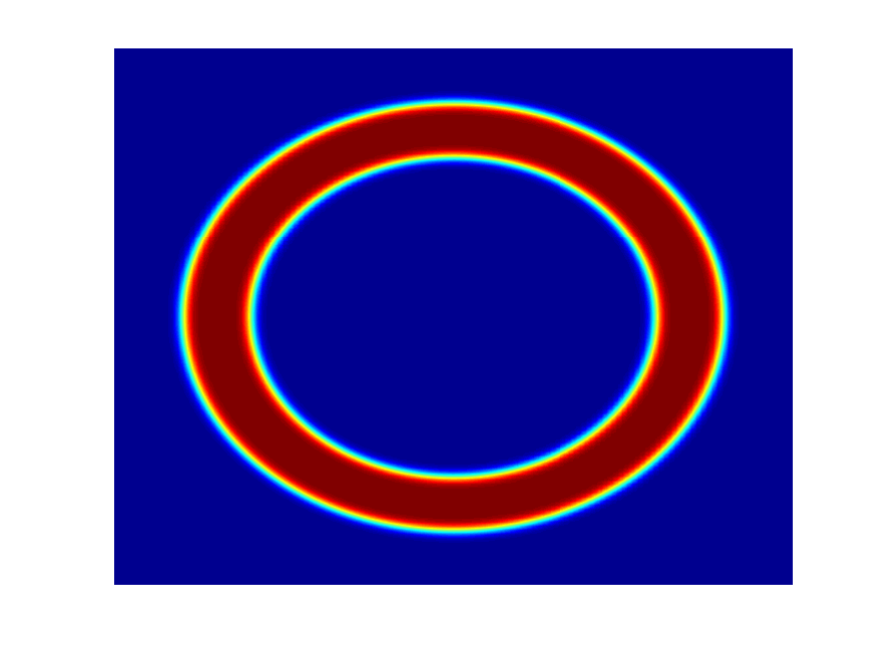}}
    \hfill
    \subfloat[t=200]{\includegraphics[width=0.33\textwidth,height=0.33\textwidth]{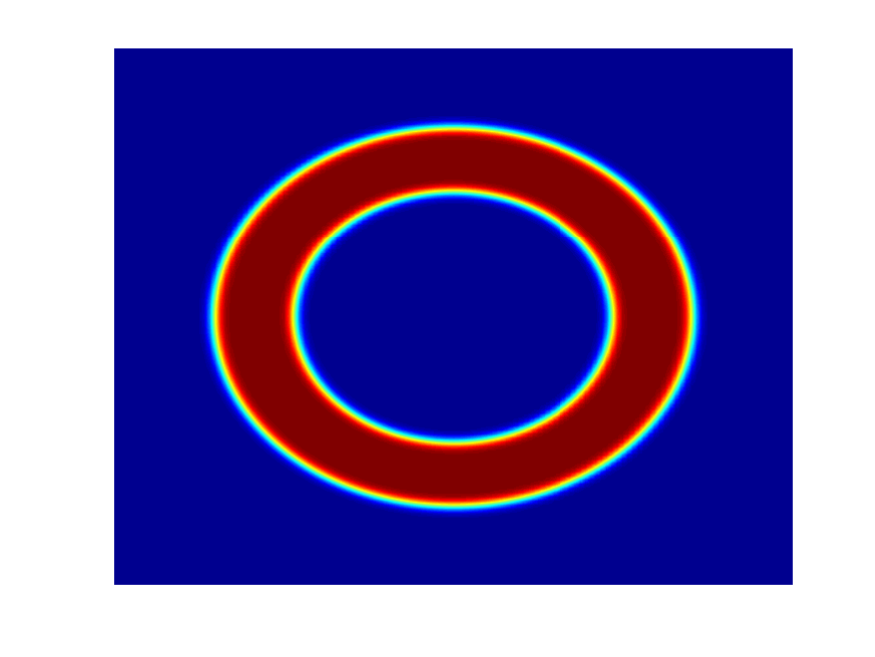}}
    \hfill
    \subfloat[t=400]{\includegraphics[width=0.33\textwidth,height=0.33\textwidth]{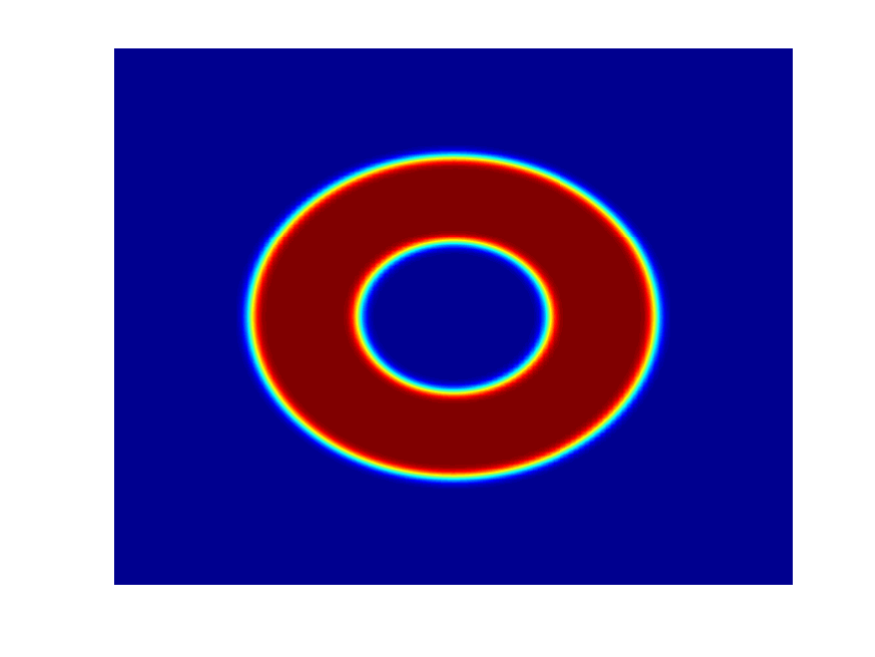}}
    \hfill
    \\
    \subfloat[t=500]{\includegraphics[width=0.33\textwidth,height=0.33\textwidth]{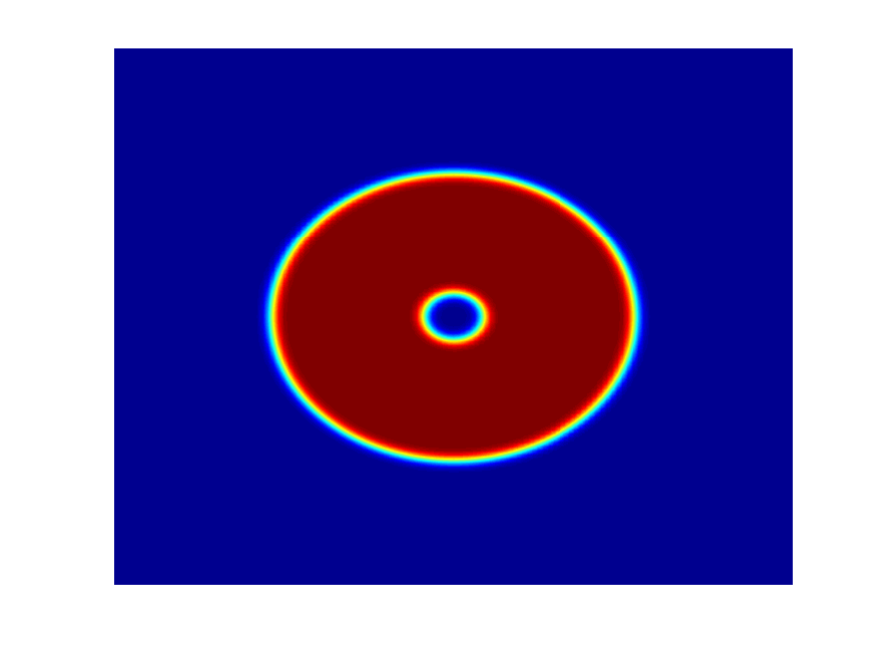}}
    \hfill
    \subfloat[t=600]{\includegraphics[width=0.33\textwidth,height=0.33\textwidth]{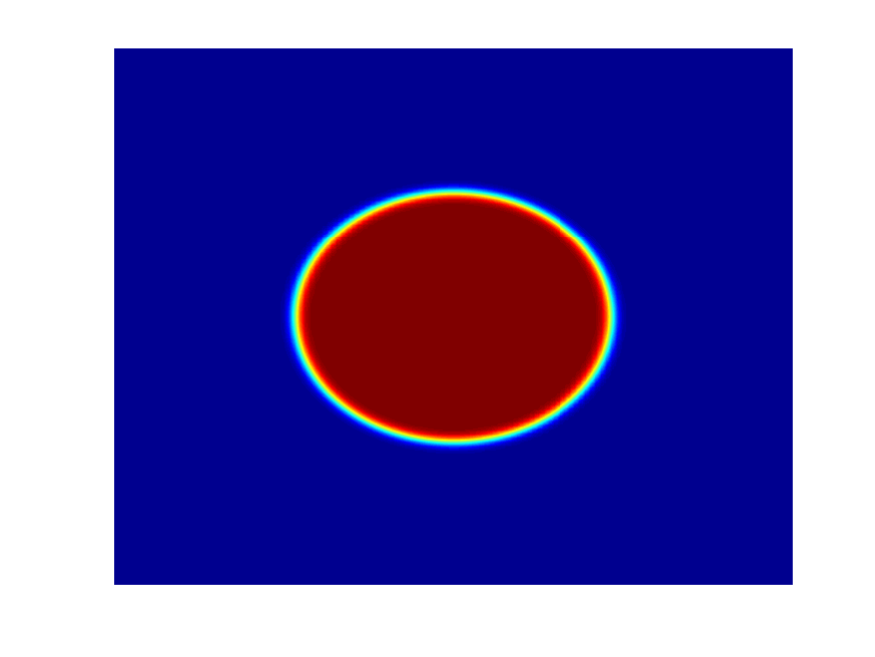}}
    \hfill
    \subfloat[t=800]{\includegraphics[width=0.33\textwidth,height=0.33\textwidth]{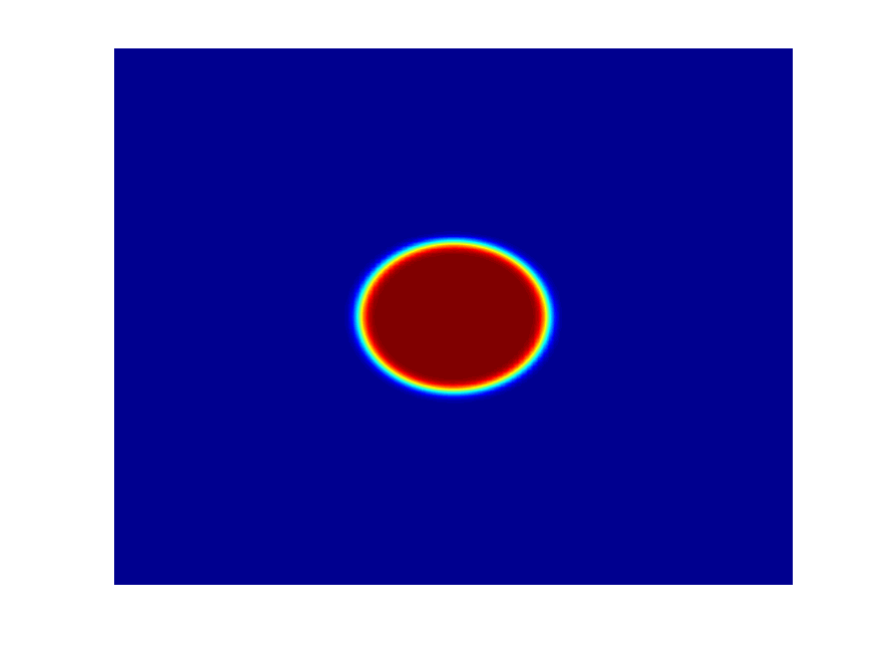}}
    \hfill
    \caption{Snapshots of the phase variable $\phi$ 
    are taken at $t=5,200,400,500,600,800$ with $\kappa =8.02,\beta=0.9575$ 
    for the Flory-Huggins potential case in \textbf{Example 2}}
    \label{F2_4}
    \hfill
\end{figure}

\begin{figure}[htp]
    \centering
    \subfloat{\includegraphics[scale=0.35]{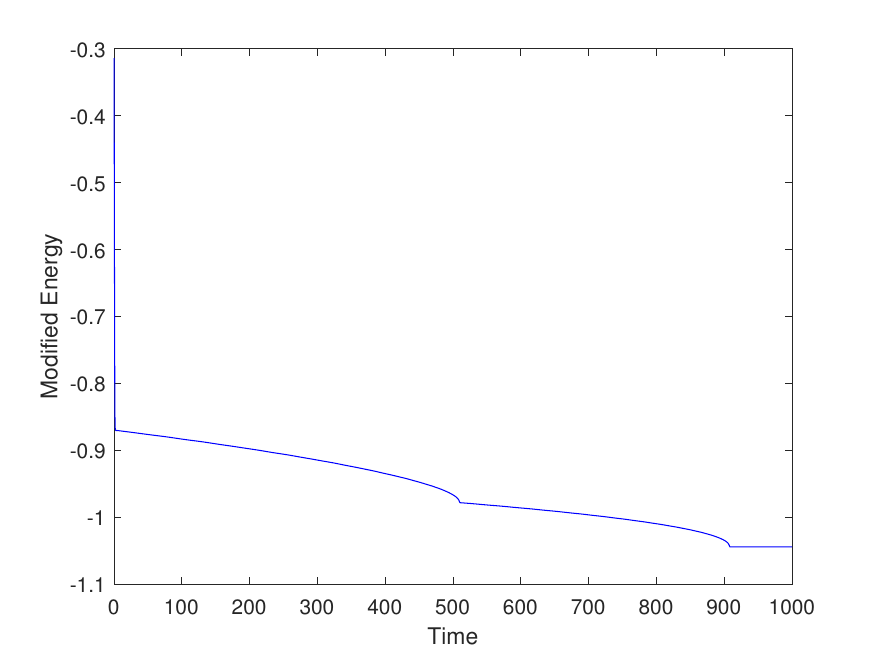}}
    \hfill
    \subfloat{\includegraphics[scale=0.35]{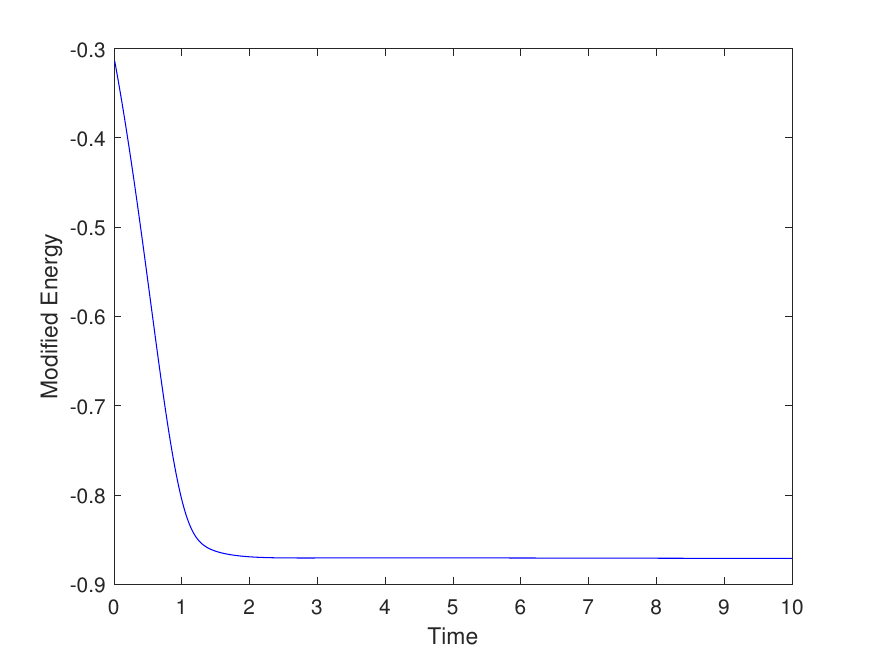}}
    \hfill
    \subfloat{\includegraphics[scale=0.35]{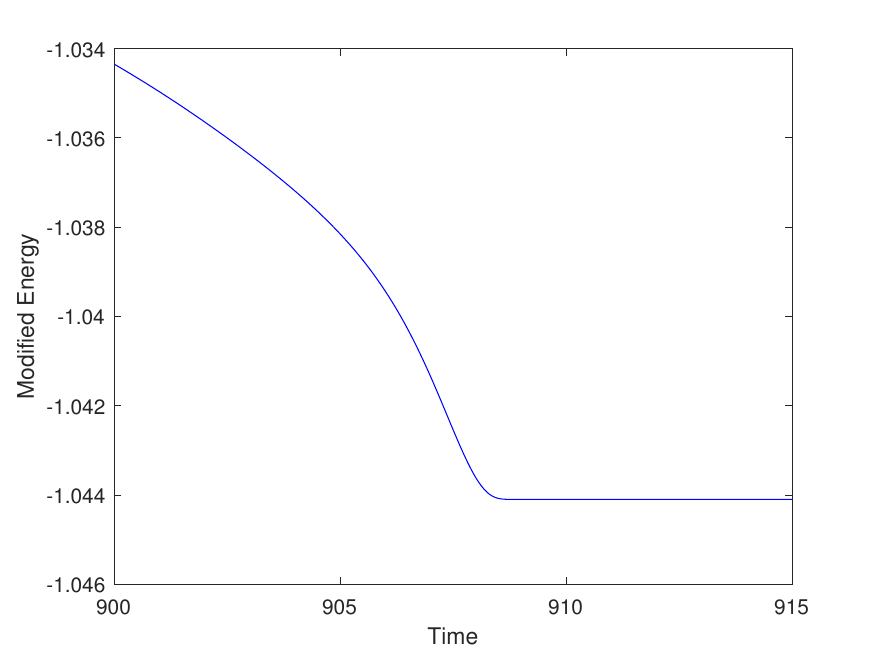}}
    \hfill
    \\
    \subfloat{\includegraphics[scale=0.35]{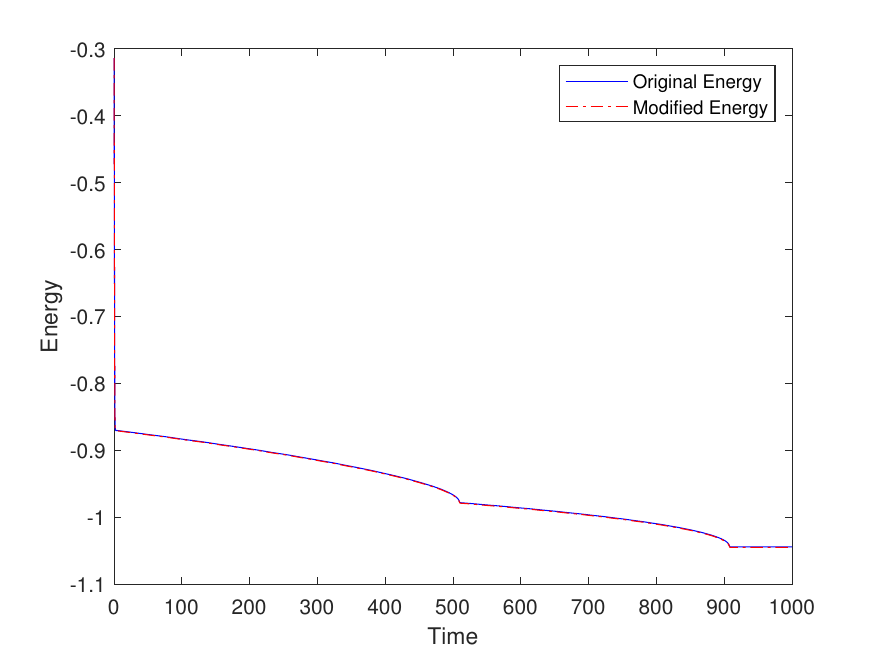}}
    \hfill
    \subfloat{\includegraphics[scale=0.35]{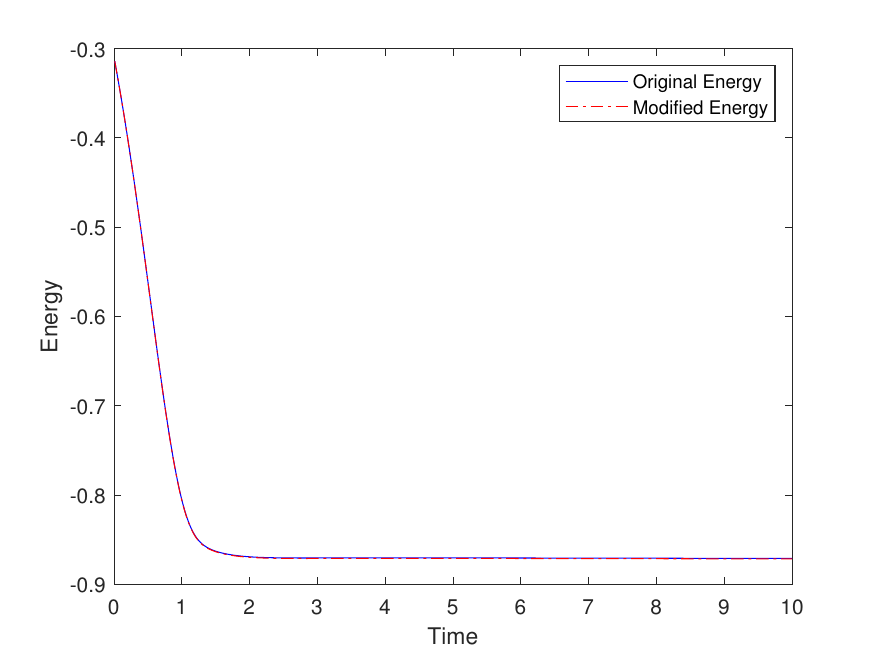}}
    \hfill
    \subfloat{\includegraphics[scale=0.35]{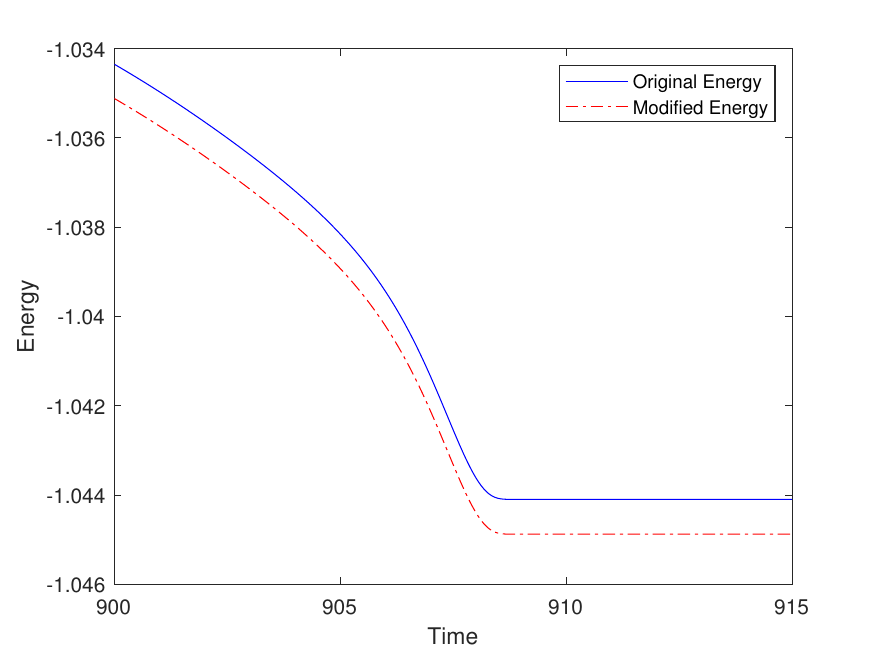}}
    \hfill
    \caption{Evolutions of the modified and original 
    free energy functional for the Flory-Huggins potential case in \textbf{Example 2}}
    \label{F2_5}
    \hfill
\end{figure}

\begin{figure}[htp]
    \centering
    \subfloat{\includegraphics[scale=0.52]{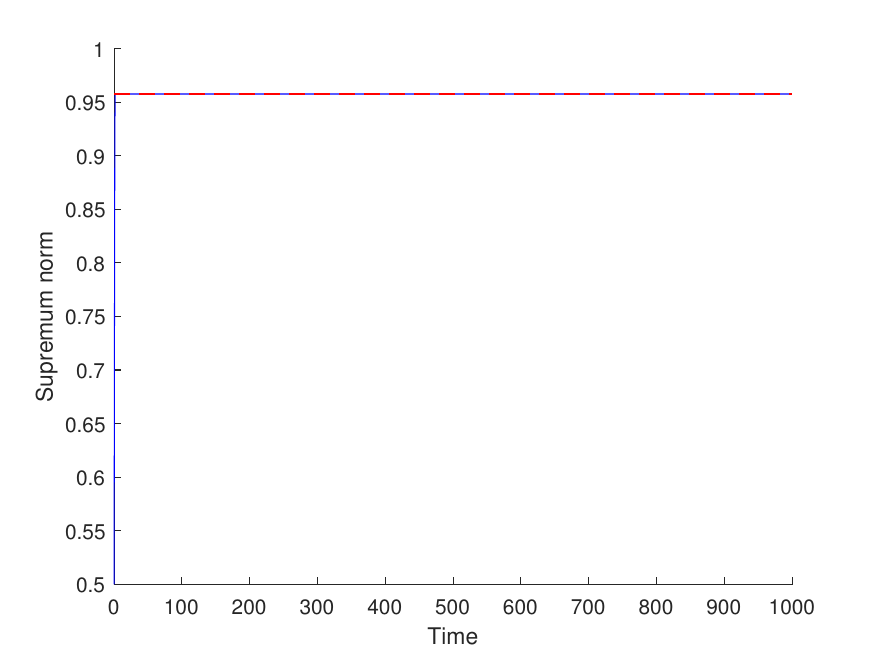}}
    \hfill
    \subfloat{\includegraphics[scale=0.52]{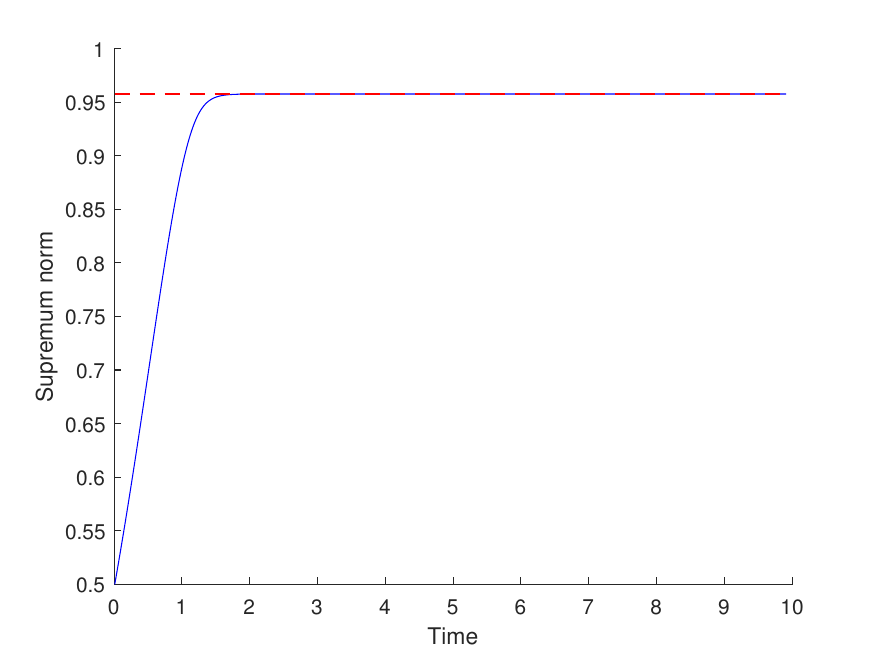}}
    \hfill
    \\
    \subfloat{\includegraphics[scale=0.52]{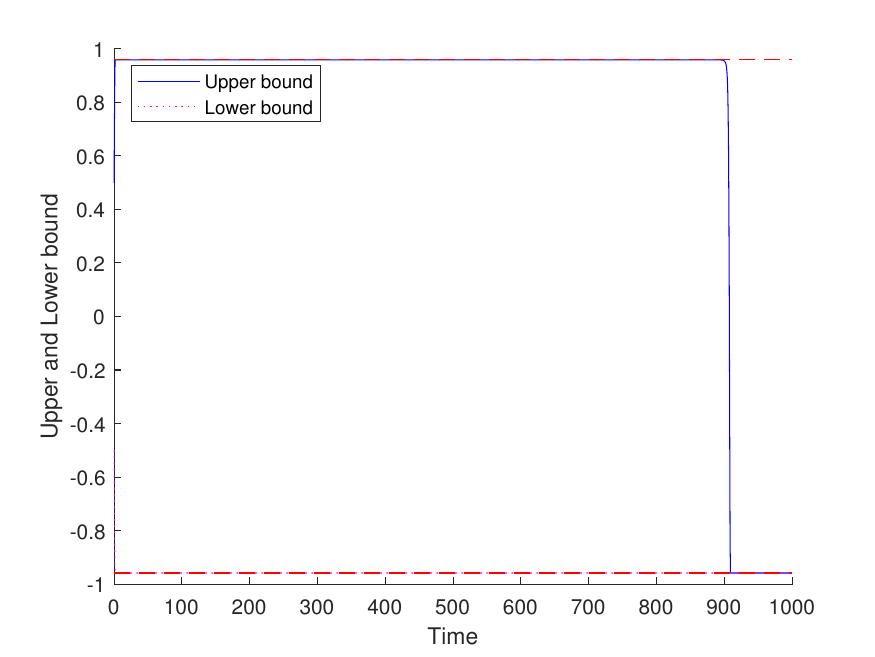}}
    \hfill
    \subfloat{\includegraphics[scale=0.52]{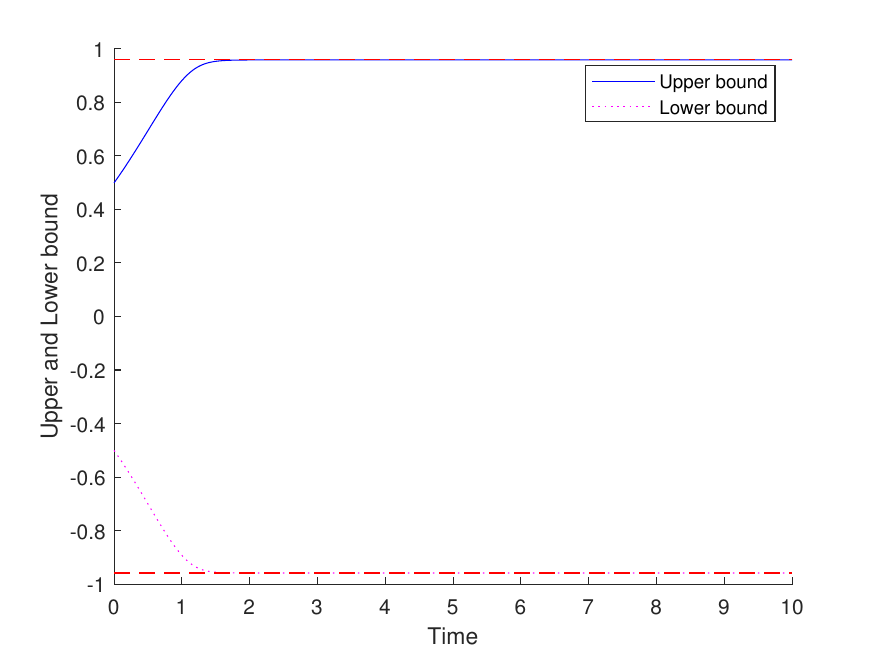}}
    \caption{Evolutions of the supremum norm (top row), 
    the upper and lower bound (bottom row) 
    of the phase variable $\phi$
    for the Flory-Huggins potential
     in \textbf{Example 2}}
    \label{F2_6}
    \hfill
\end{figure}

\textbf{Example 3}(Long-time coarsening dynamics sumulations)
Now we discuss the coarsening dynamics driven by the NAC equation \eqref{eqNAC}
with $N=2^7,\tau=0.01,\varepsilon = 0.02,\delta =0.02.$
The initial condition is given by 
$$
\phi \left( x,y,0 \right) =0.5rand(x,y),  
$$
where $rand(x,y)$ is a random number in $[-0.5,0.5]$ with zero mean.
We use the sESAV2 scheme to simulate the long-time coarsening process.
\par
For the double-well potential case \eqref{4-1}, we take $\kappa=2,\beta=1$, 
the snapshots of the phase structure captured at $t=0,6,16,50,350,700$ are 
shown in Figure \ref{F3_1}, 
and the steady state $\phi \equiv -1 $ arises at about $t=1855$.
The preservation of MBP during the whole phase transformation process can be 
observed from Figure \ref{F3_3}. The evolution of energy is plotted in 
Figure \ref{F3_2}, illustrating that this process maintains 
energy dissipation.
\par
We take $\kappa=8.02,\beta=0.9575$ for the Flory-Huggins potential case \eqref{4-2}
and the numerical results are presented in Figure \ref{F3_4}. 
From Ficture \ref{F3_5}, we can see that the model reaches a steady state at about $t = 1866$, 
and the whole phase separation progress as in the case of the double-well potential 
situation. 
To conclude, these results are almost identical to those of \cite{10,43},
and we find that We also observe that the phase residue $\phi$ 
corresponding to the NAC models evolves more slowly than the classical Allen-Cahn equation.

\begin{figure}[htp]
    \centering
    \subfloat[t=0]{\includegraphics[width=0.33\textwidth,height=0.33\textwidth]{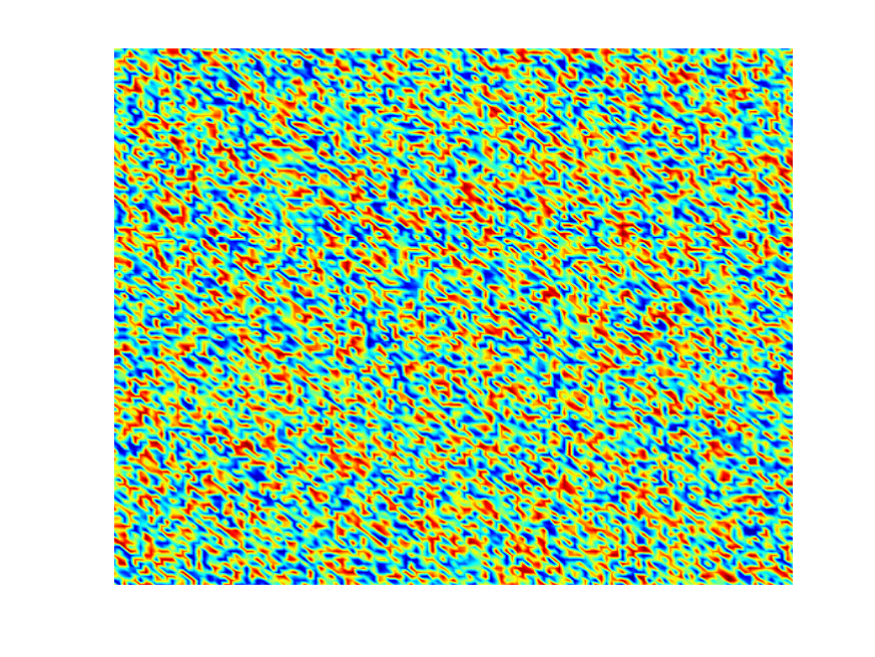}}
    \hfill
    \subfloat[t=6]{\includegraphics[width=0.33\textwidth,height=0.33\textwidth]{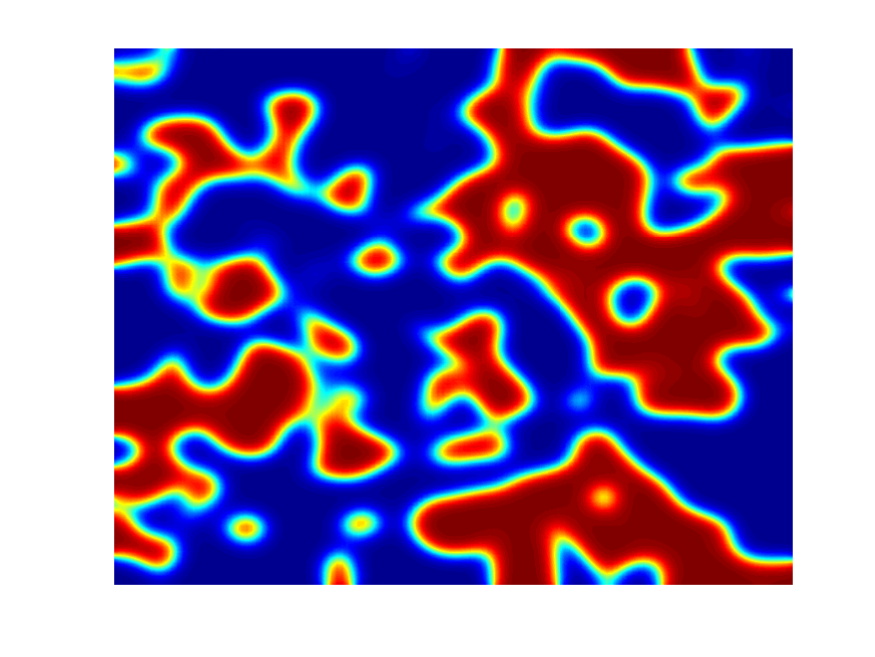}}
    \hfill
    \subfloat[t=16]{\includegraphics[width=0.33\textwidth,height=0.33\textwidth]{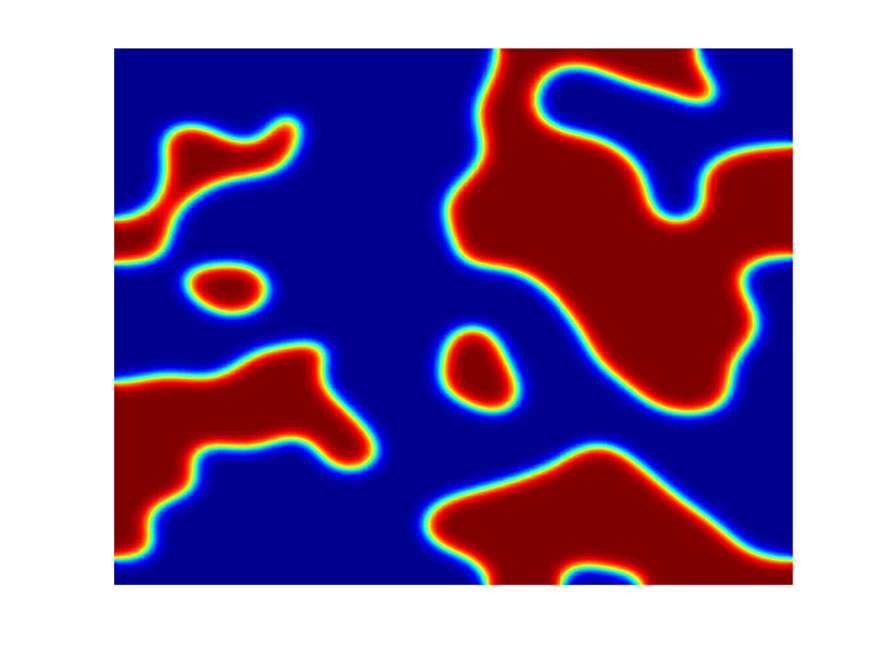}}
    \hfill
    \\
    \subfloat[t=50]{\includegraphics[width=0.33\textwidth,height=0.33\textwidth]{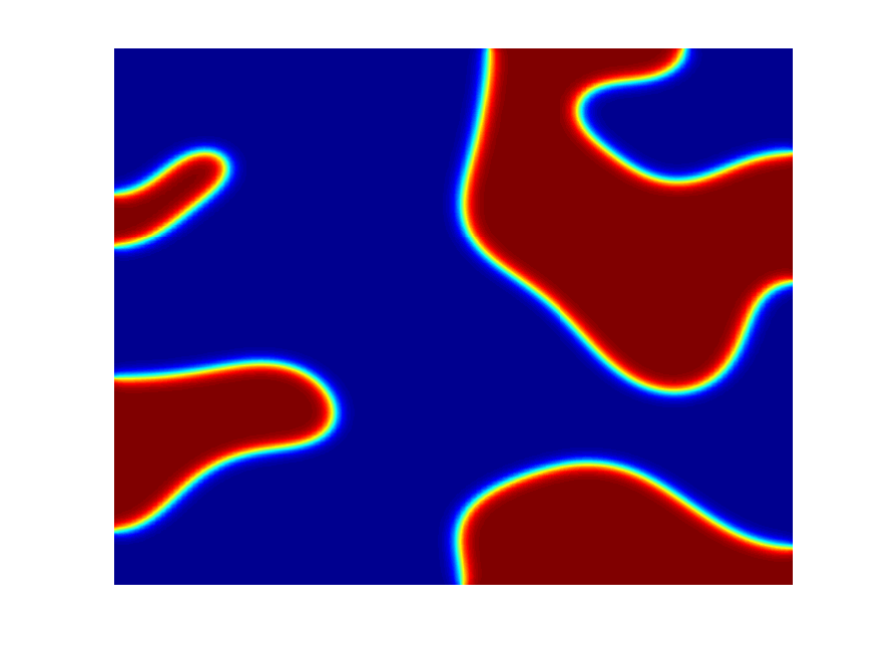}}
    \hfill
    \subfloat[t=350]{\includegraphics[width=0.33\textwidth,height=0.33\textwidth]{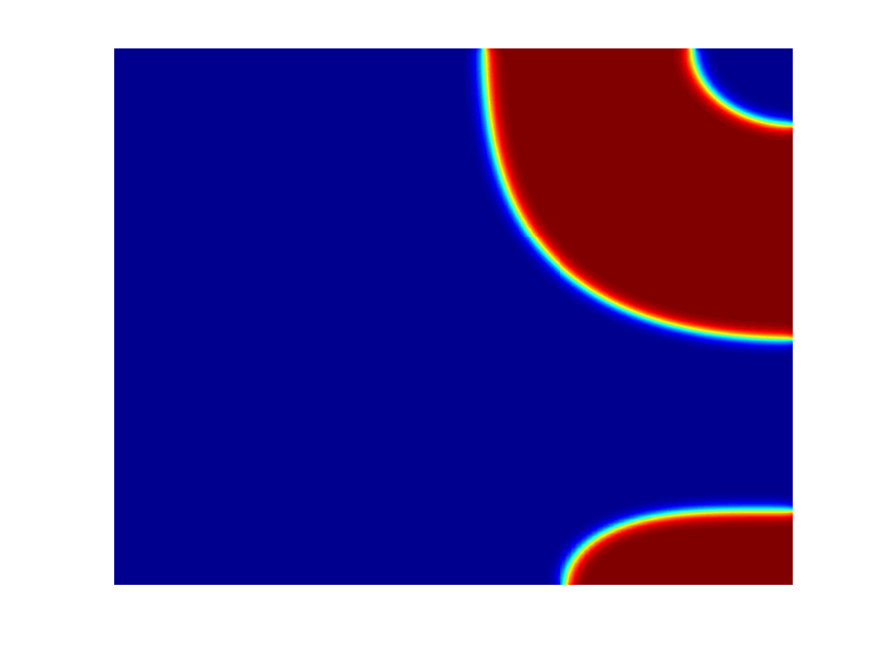}}
    \hfill
    \subfloat[t=700]{\includegraphics[width=0.33\textwidth,height=0.33\textwidth]{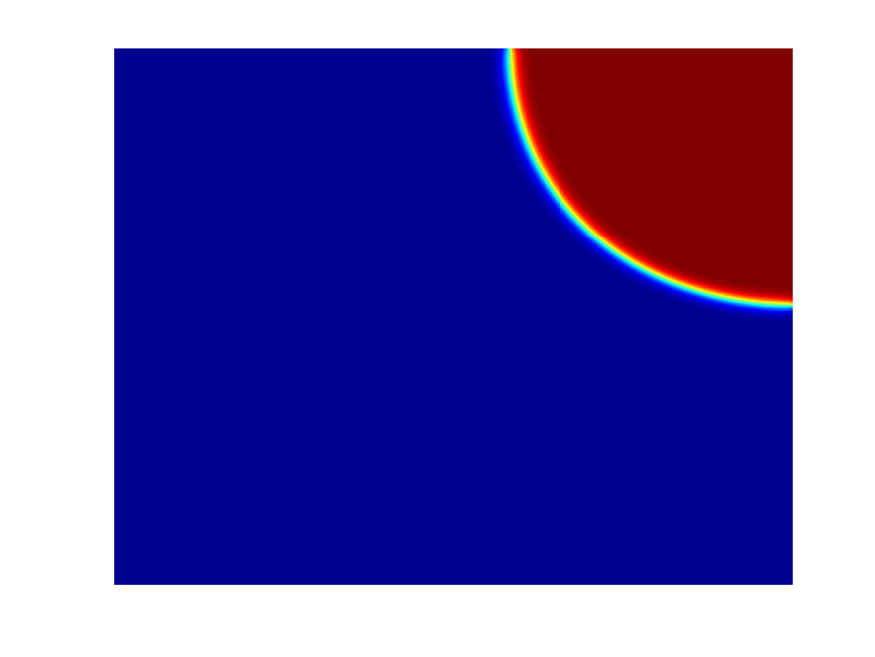}}
    \hfill
    \caption{Snapshots of the phase variable $\phi$ 
    are taken at $t=0,6,16,50,350,700$ with $\kappa =2,\beta=1$ 
    for the double-well potential case in \textbf{Example 3}}
    \label{F3_1}
    \hfill
\end{figure}

\begin{figure}[htp]
    \centering
    \subfloat{\includegraphics[scale=0.35]{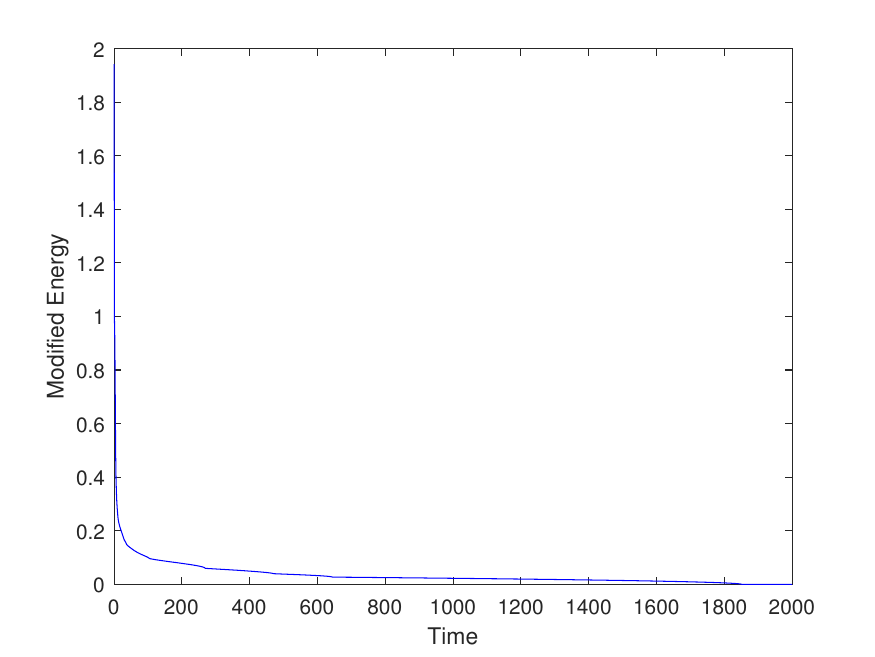}}
    \hfill
    \subfloat{\includegraphics[scale=0.35]{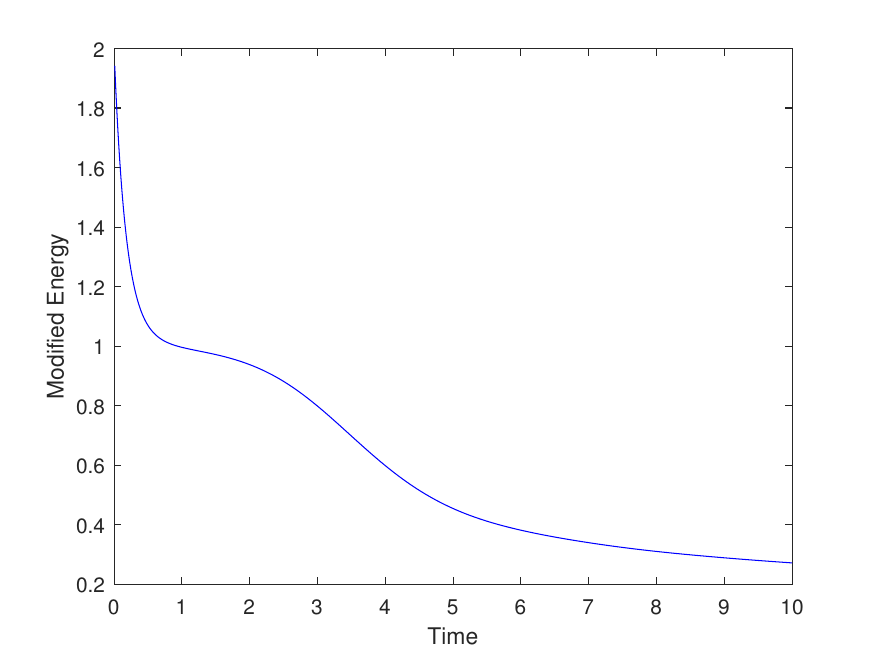}}
    \hfill
    \subfloat{\includegraphics[scale=0.35]{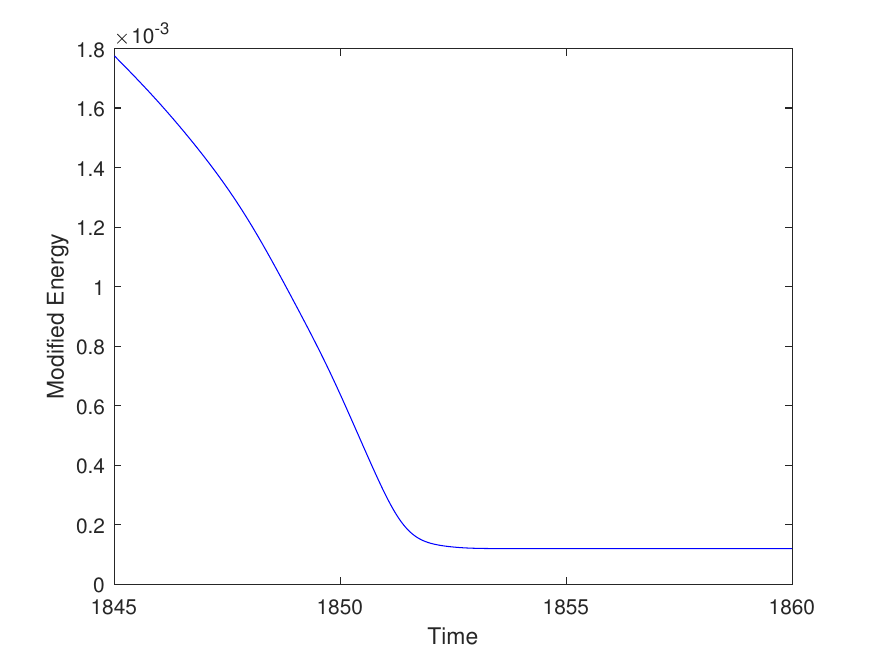}}
    \hfill
    \\
    \subfloat{\includegraphics[scale=0.35]{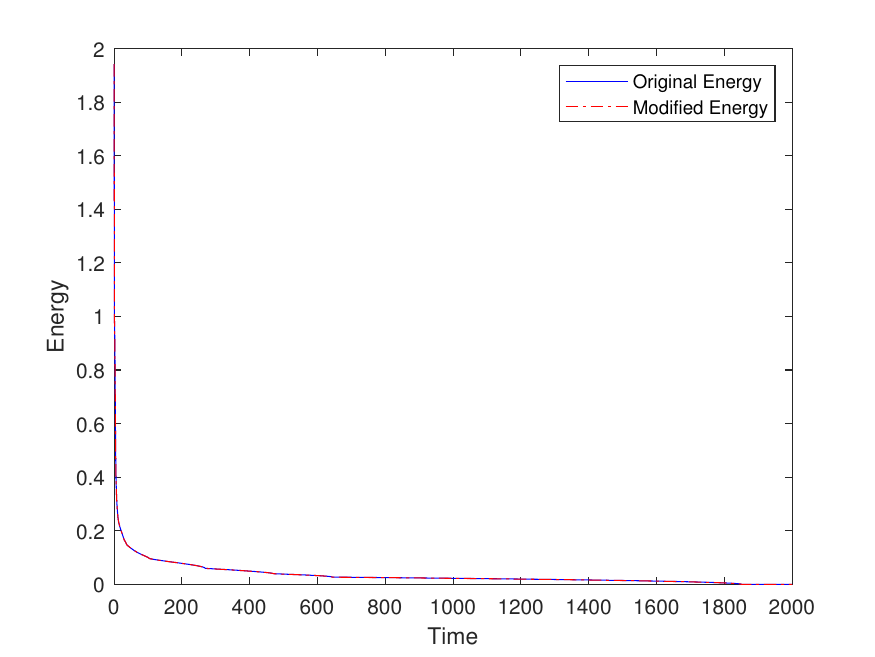}}
    \hfill
    \subfloat{\includegraphics[scale=0.35]{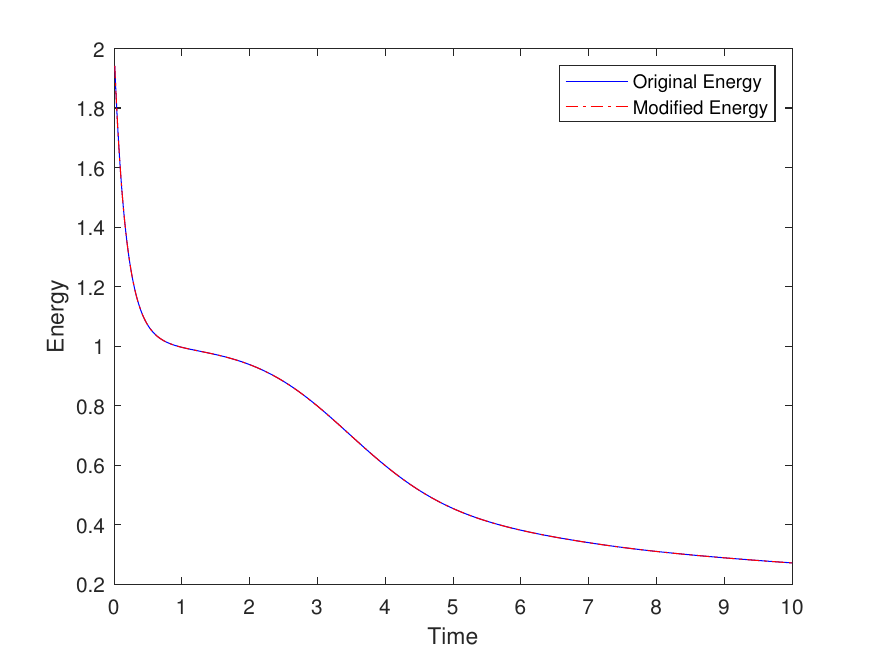}}
    \hfill
    \subfloat{\includegraphics[scale=0.35]{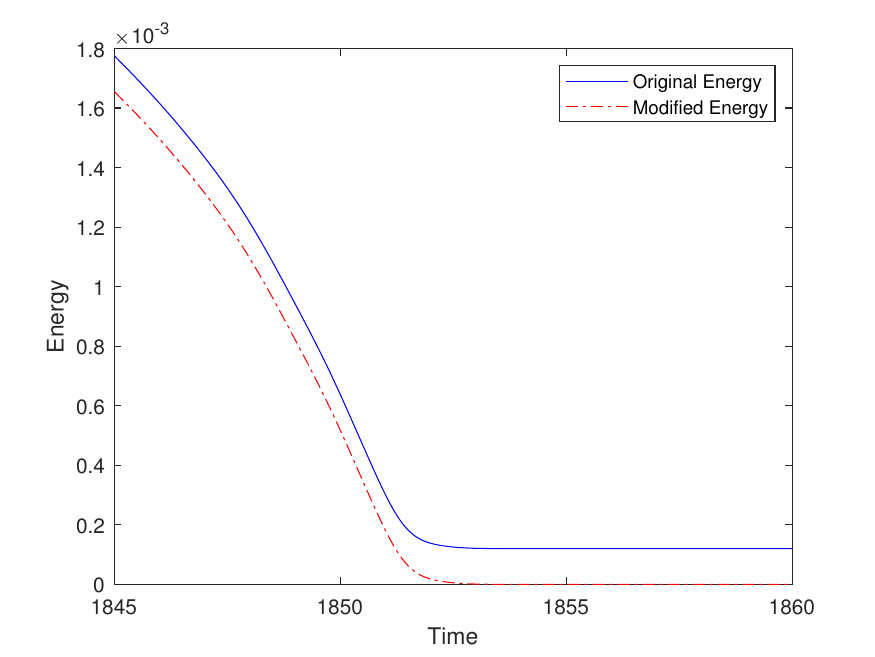}}
    \hfill
    \caption{Evolutions of the modified and original free energy 
    functional for the double-well potential case in \textbf{Example 3}}
    \label{F3_2}
    \hfill
\end{figure}

\begin{figure}[htp]
    \centering
    \subfloat{\includegraphics[scale=0.52]{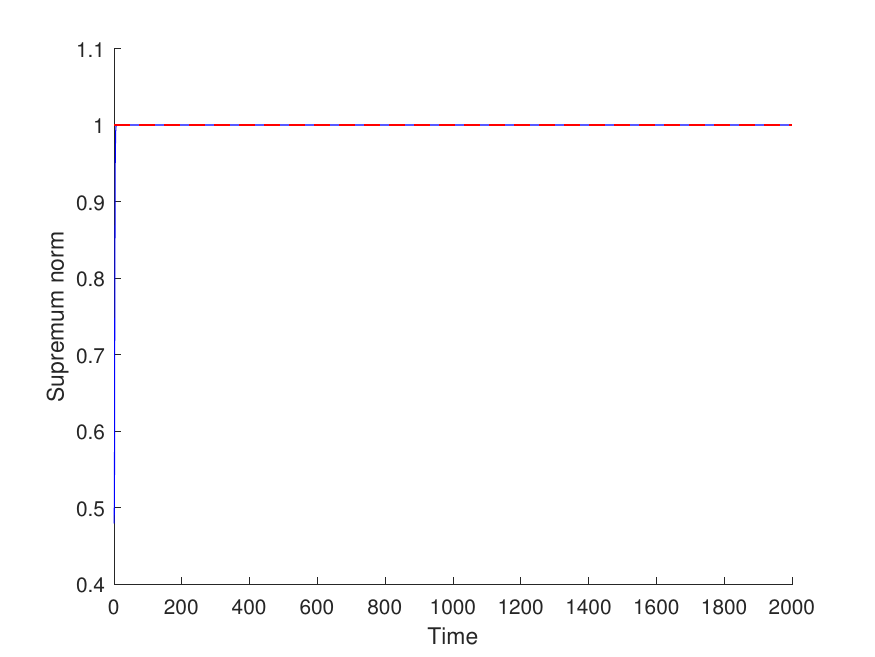}}
    \hfill
    \subfloat{\includegraphics[scale=0.52]{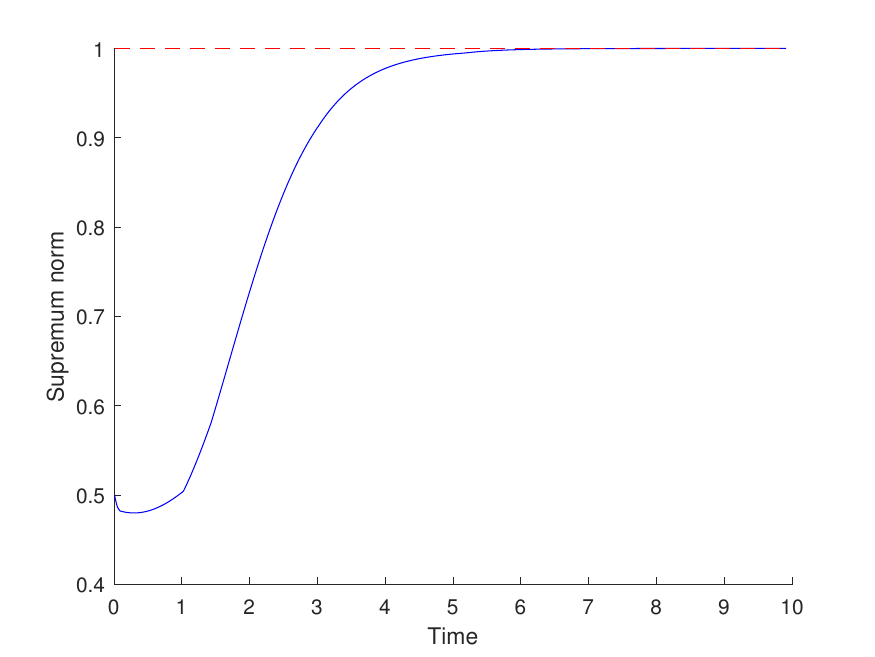}}
    \hfill
    \\
    \subfloat{\includegraphics[scale=0.52]{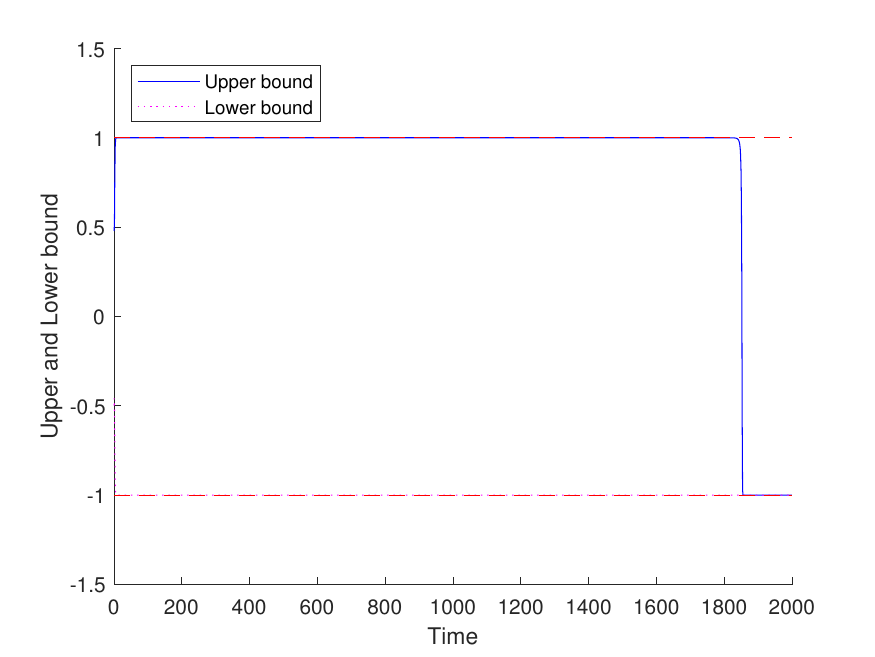}}
    \hfill
    \subfloat{\includegraphics[scale=0.52]{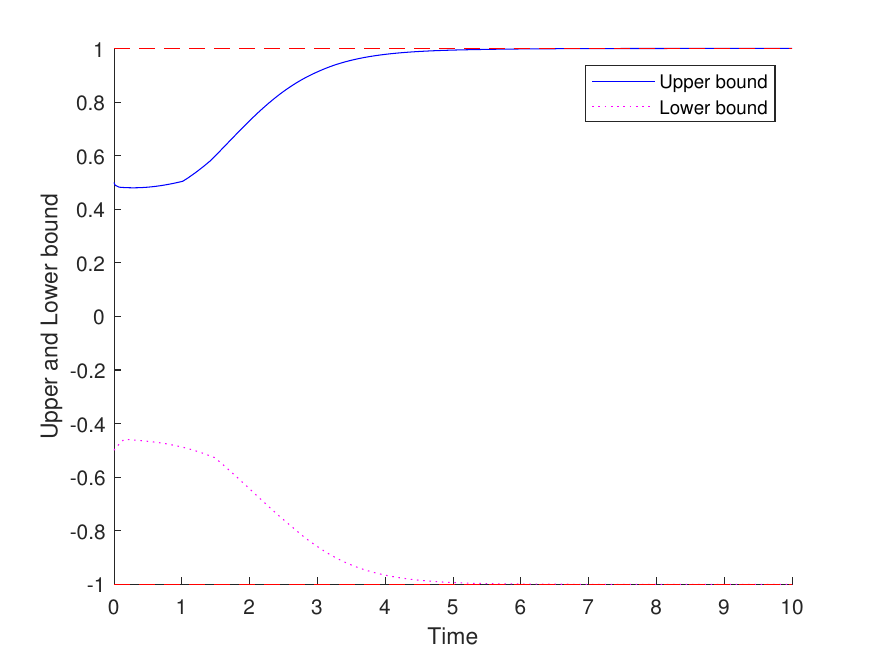}}
    \caption{Evolutions of the supremum norm (top row), 
    the upper and lower bound (bottom row) 
    of the phase variable $\phi$ for the double-well potential in \textbf{Example 3}}
    \label{F3_3}
    \hfill
\end{figure}

\begin{figure}[htp]
    \centering
    \subfloat[t=0]{\includegraphics[width=0.33\textwidth,height=0.33\textwidth]{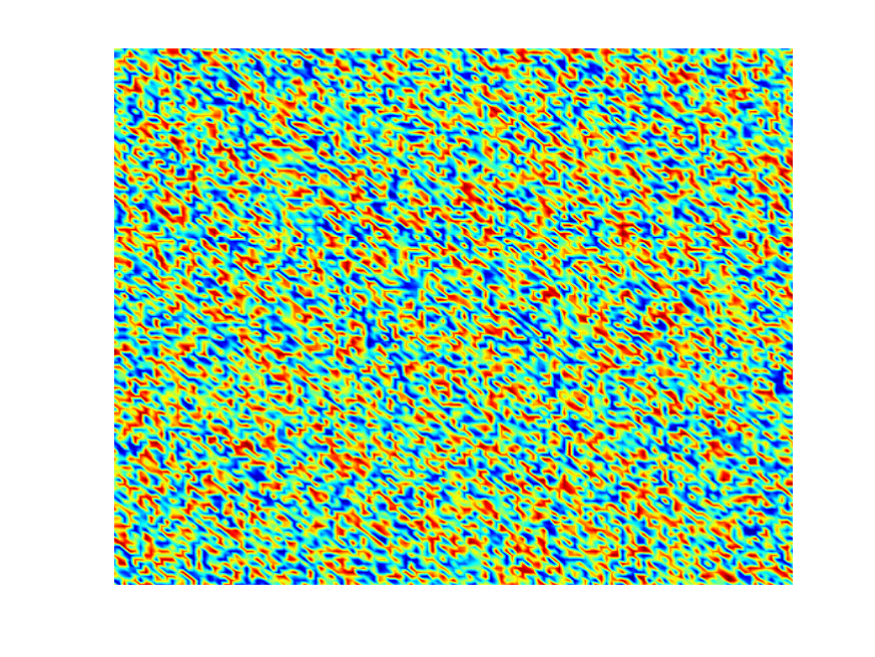}}
    \hfill
    \subfloat[t=6]{\includegraphics[width=0.33\textwidth,height=0.33\textwidth]{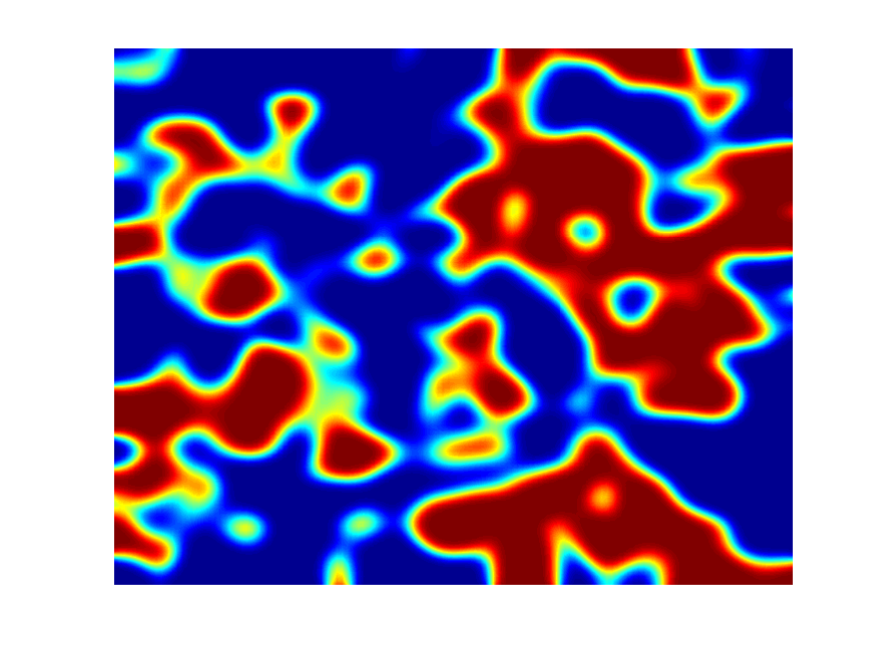}}
    \hfill
    \subfloat[t=16]{\includegraphics[width=0.33\textwidth,height=0.33\textwidth]{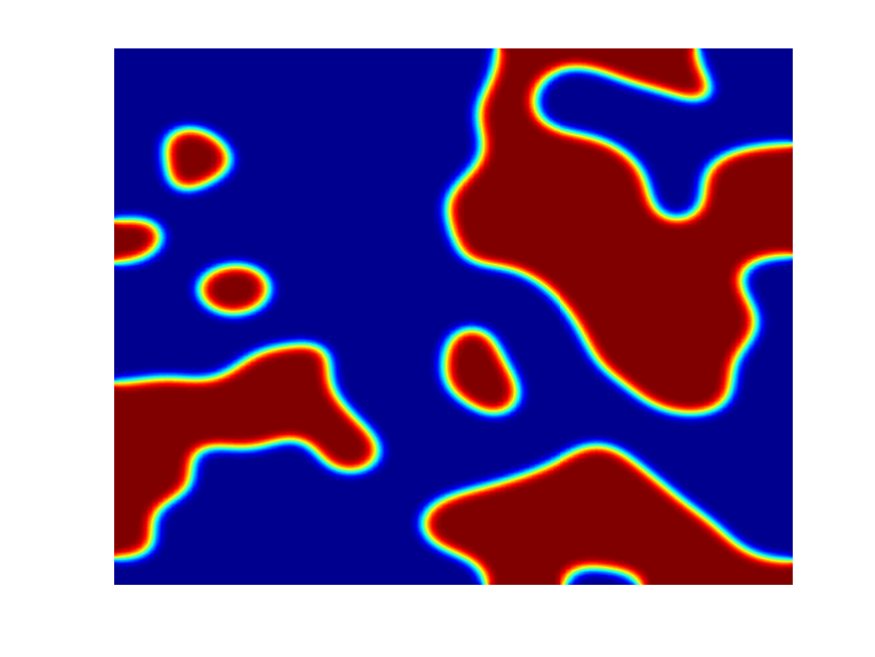}}
    \hfill
    \\
    \subfloat[t=50]{\includegraphics[width=0.33\textwidth,height=0.33\textwidth]{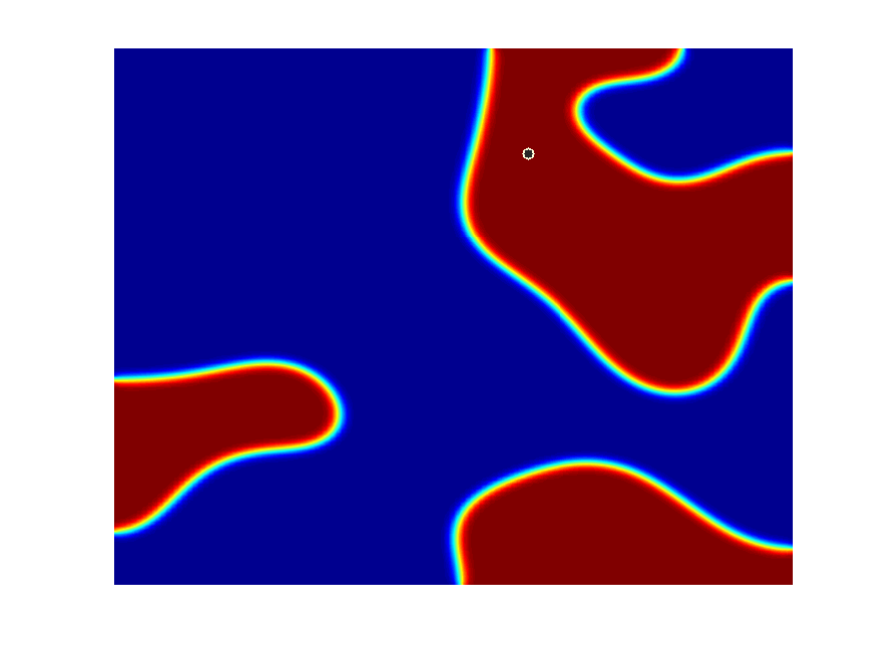}}
    \hfill
    \subfloat[t=350]{\includegraphics[width=0.33\textwidth,height=0.33\textwidth]{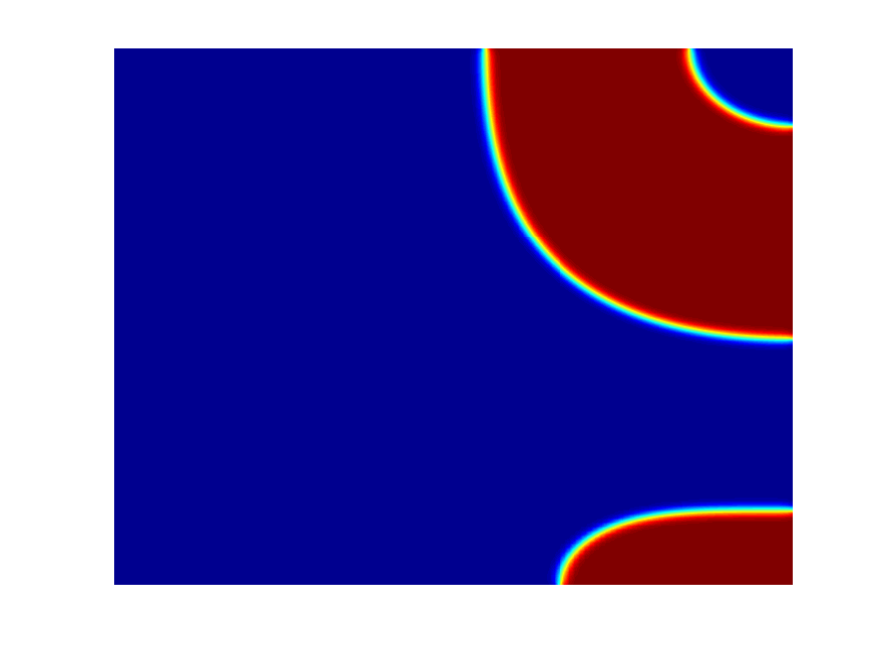}}
    \hfill
    \subfloat[t=700]{\includegraphics[width=0.33\textwidth,height=0.33\textwidth]{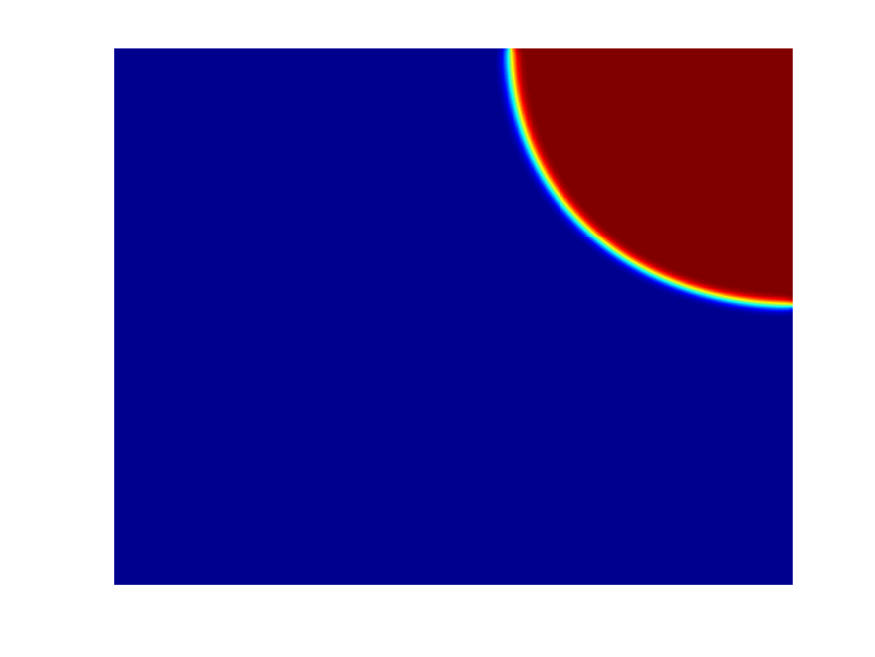}}
    \hfill
    \caption{Snapshots of the phase variable $\phi$ 
    are taken at $t=0,6,16,50,350,700$ with $\kappa =8.02,\beta=0.9575$ 
    for the Flory-Huggins potential case in \textbf{Example 3}}
    \label{F3_4}
    \hfill
\end{figure}

\begin{figure}[htp]
    \centering
    \subfloat{\includegraphics[scale=0.35]{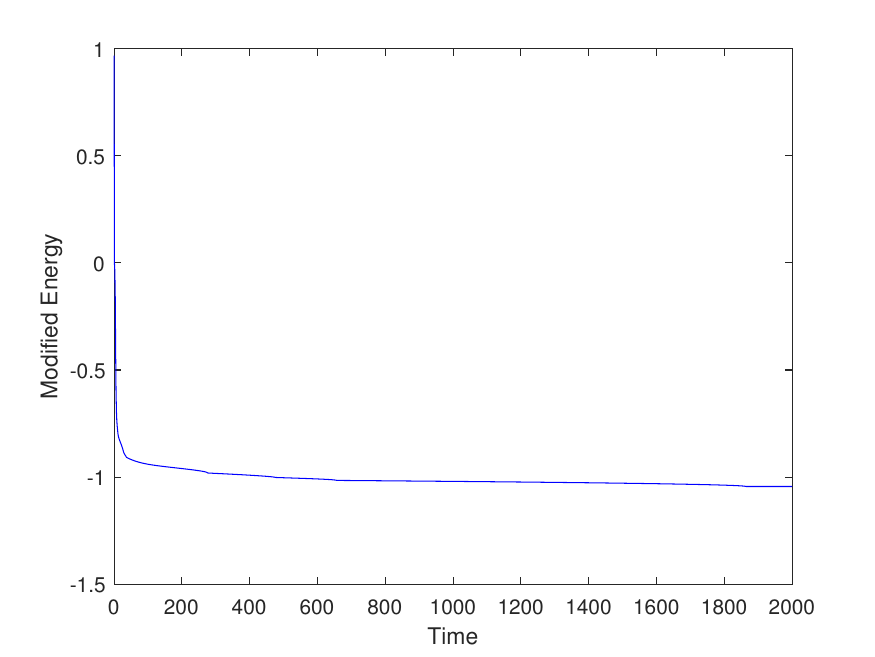}}
    \hfill
    \subfloat{\includegraphics[scale=0.35]{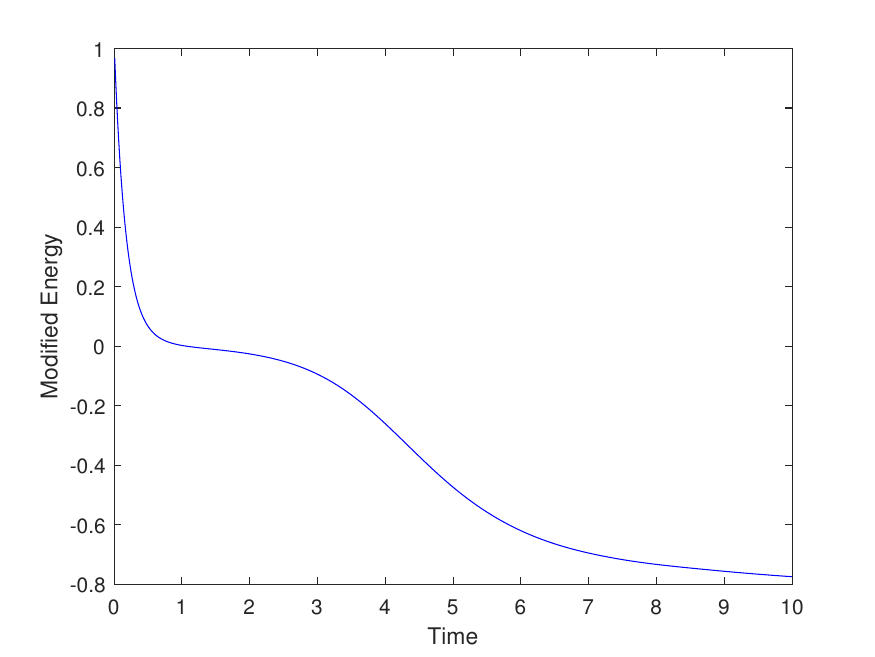}}
    \hfill
    \subfloat{\includegraphics[scale=0.35]{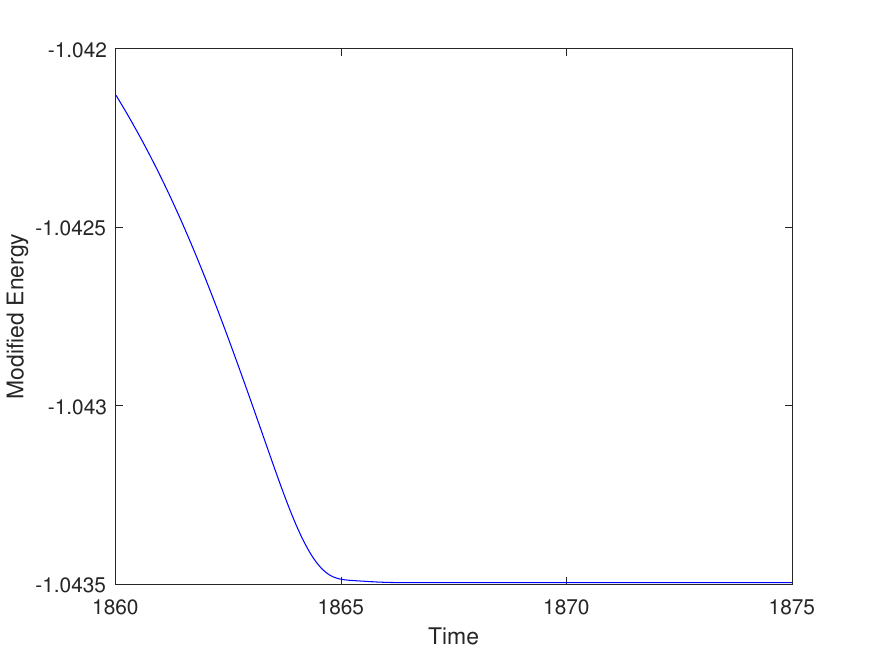}}
    \hfill
    \\
    \subfloat{\includegraphics[scale=0.35]{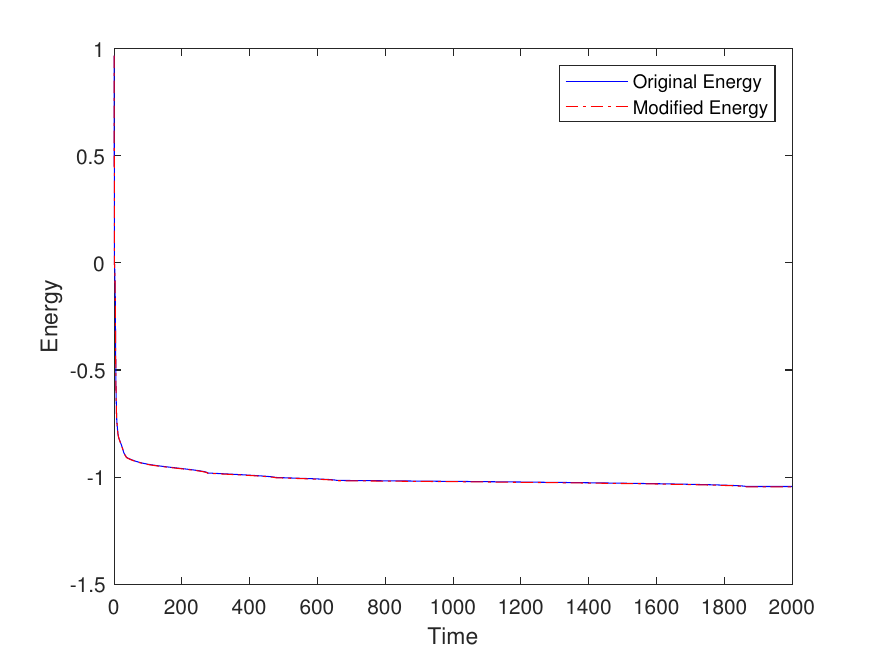}}
    \hfill
    \subfloat{\includegraphics[scale=0.35]{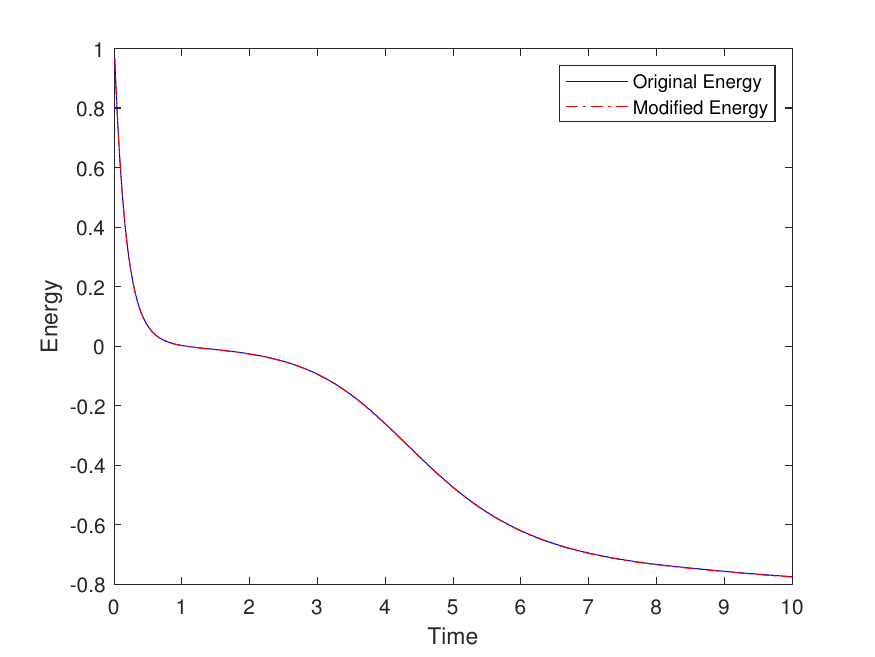}}
    \hfill
    \subfloat{\includegraphics[scale=0.35]{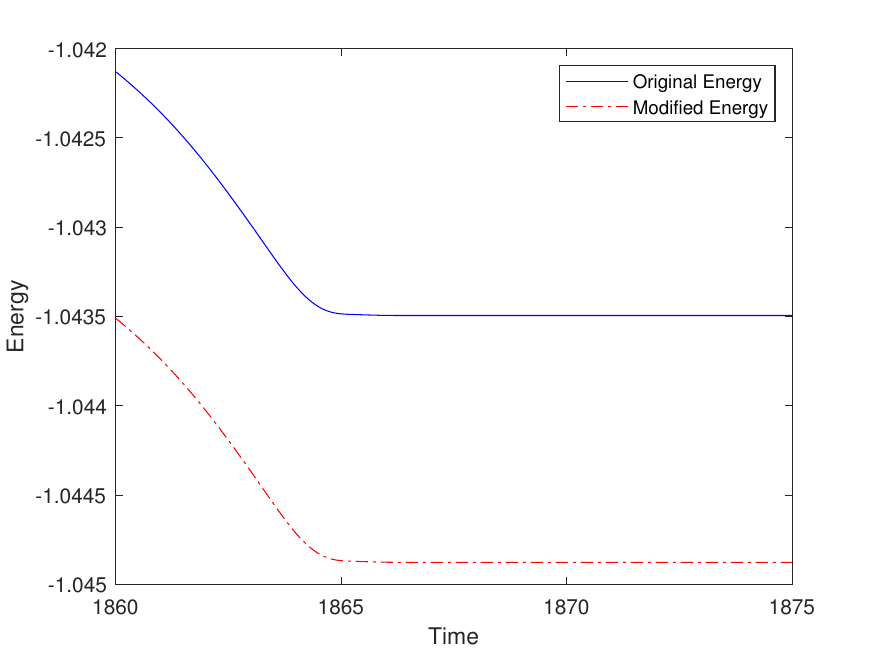}}
    \hfill
    \caption{Evolutions of the modified and original free energy 
    functional for the Flory-Huggins potential case in \textbf{Example 3}}
    \label{F3_5}
    \hfill
\end{figure}

\begin{figure}[htp]
    \centering
    \subfloat{\includegraphics[scale=0.52]{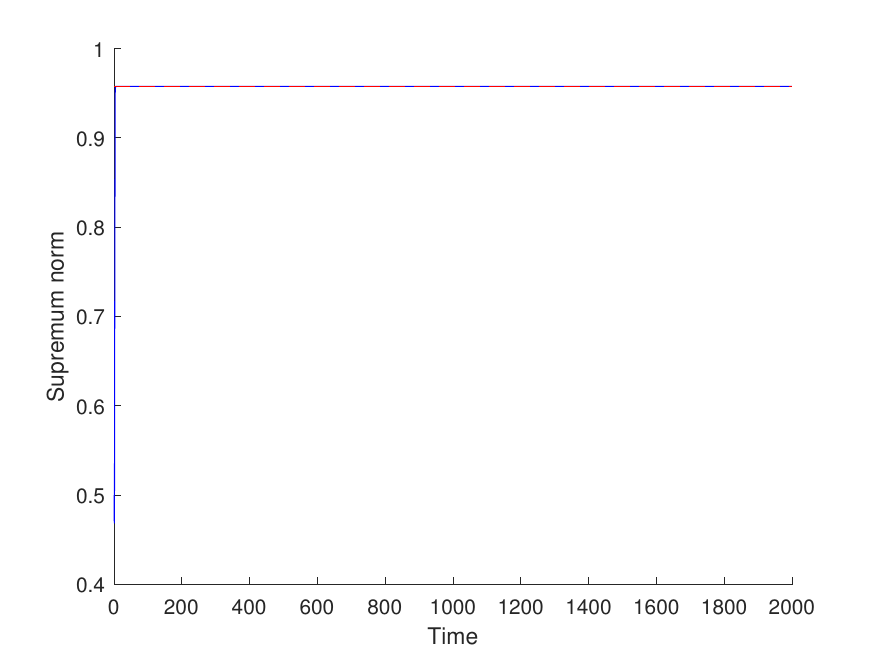}}
    \hfill
    \subfloat{\includegraphics[scale=0.52]{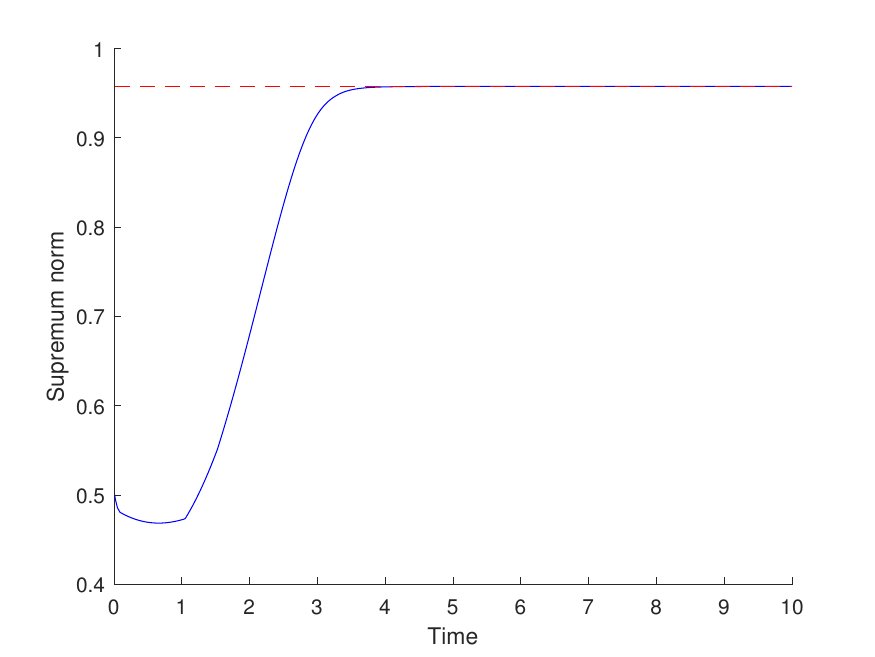}}
    \hfill
    \\
    \subfloat{\includegraphics[scale=0.52]{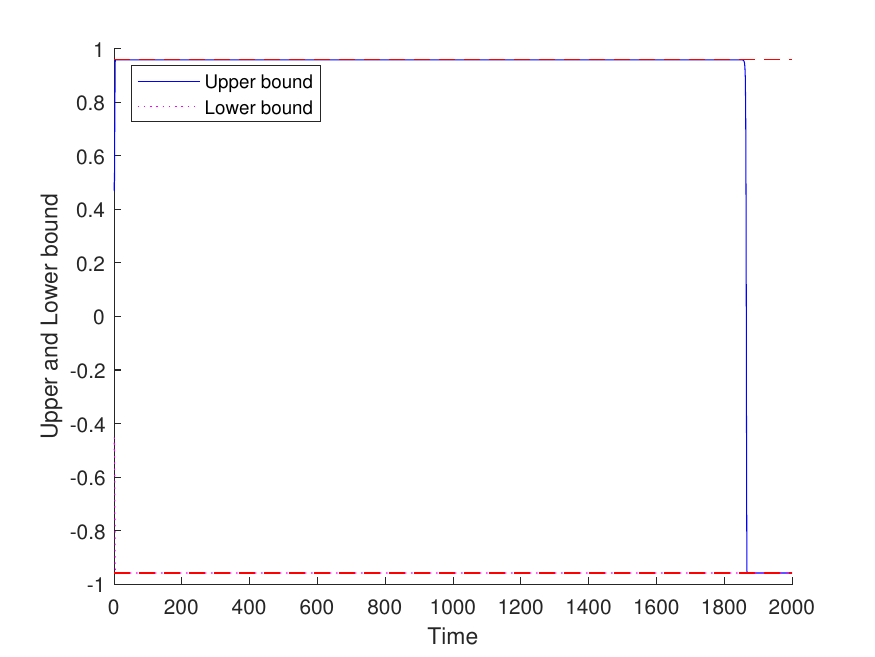}}
    \hfill
    \subfloat{\includegraphics[scale=0.52]{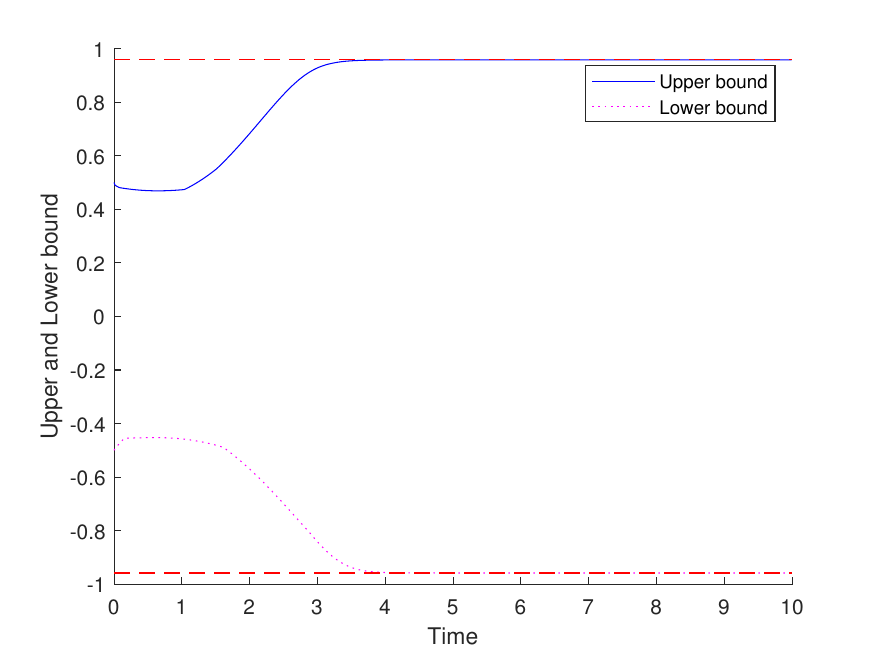}}
    \caption{Evolutions of the supremum norm (top row), 
    the upper and lower bound (bottom row) 
    of the phase variable $\phi$ for the Flory-Huggins potential in \textbf{Example 3}}
    \label{F3_6}
    \hfill
\end{figure}
\section{Conclusions}
In this paper, we construct first- and second-order (in time) 
efficient and linear schemes based on sESAV approach \cite{10}
for the NAC equation with general nonlinear potential. 
More importantly, both schemes satisfy unconditional energy stability and MBP. 
In addition, for MBP preservation, the first-order scheme is unconditional, 
while the second-order scheme has a limit on the time step.
We use finite difference method to discretize the nonlocal 
operator, and a fast solver based on FFT is implemented
for acceleration of the procedure. 
By analyzing and processing the resulting 
stiffness matrices and combining the relevant lemmas, we 
carefully and rigorously prove that the two proposed 
numerical schemes satisfy unconditional energy stability and 
MBP at the fully discrete level.
Numerical experiments are carried out using the $\delta $-parameterized 
Gaussian kernel to verify the optimal temporal convergence rate and
to observe the steady phase transition interface.
The numerical results show that the evolution results presented by NAC model
are consistent with the numerical results of classical Allen-Cahn model 
under the same conditions. 
The difference is that it takes a longer time for NAC model to reach steady state.
The above conclusions show that the proposed numerical schemes is effective and efficient.\\

\noindent \textbf{Ethical Approval} \, Not applicable.\\

\noindent \textbf{Availability of supporting data} \, Data sharing not applicable to this article as no datasets were generated or analyzed during the current study.\\

\noindent \textbf{Competing interests} \, The authors declare no competing interests.\\

\noindent \textbf{Funding} \, This work was supported by the National Natural Science Foundation of China under Grants 11971272, 12001336.\\

\noindent \textbf{Acknowledgement} \, Not applicable. \\

\noindent \textbf{Authors' contributions} \, Xiaoqing Meng wrote the main manuscript text and completed the numerical experiments. Aijie Cheng, Zhengguang Liu and Xiaoqing Meng carried out the numerical analysis. All authors reviewed the manuscript.\\

    \bibliographystyle{plain}
    \bibliography{main}

\end{document}